\documentclass[a4paper,10pt]{article} 

\usepackage{bbm}
\usepackage[dvips]{graphics} 
\usepackage[usenames]{color}
\usepackage{theorem}
\usepackage{amssymb, amsmath, amsbsy, amsfonts,stmaryrd}
\usepackage{graphics,graphicx}
\usepackage{mathtools}  
\mathtoolsset{showonlyrefs}

\newtheorem{theorem}{Theorem}[section]
\newtheorem{lemma}[theorem]{Lemma}
\newtheorem{proposition}[theorem]{Proposition}
\newtheorem{definition}[theorem]{Definition}
\newtheorem{corollary}[theorem]{Corollary}
{\theorembodyfont{\rmfamily}

\newtheorem{remark}[theorem]{Remark}
}

\newcommand{\transposee}[1]{{\vphantom{#1}}^{t}{#1}}

\newcommand{\N}{\mathbb{N}}
\newcommand{\Z}{\mathbb{Z}}
\newcommand{\R}{\mathbb{R}}

\newcommand{\T}{\mathbb{T}}

\newcommand{\n}{\boldsymbol{n}}
\renewcommand{\j}{\textsf{j}}
\newcommand{\B}{\boldsymbol{B}}
\renewcommand{\div}{{{\rm div}}}
\newcommand{\Ker}{{{\rm Ker}}}
\newcommand{\divt}{{{\rm div}_{\tau}}}
\newcommand{\nablat}{{{\nabla}_{\tau}}}
\newcommand{\Deltat}{{{\Delta}_{\tau}}}
\newcommand{\sH}{{\rm H}}
\newcommand{\sW}{{\rm W}}
\newcommand{\sL}{{\rm L}}
\newcommand{\sC}{{\rm C}}
\newcommand{\C}{{C}}

\newcommand{\J}{\mathcal{J}}
\newcommand{\A}{\mathcal{A}}

\newcommand{\V}{\mathcal{V}}
\renewcommand{\H}{\mathcal{H}}
\newcommand{\He}{\mathsf{H}}

\newcommand{\Om}{\Omega}

\newcommand{\Id}{I\!d}

\renewcommand{\Bar}{{\rm Bar}}

\newcommand{\Vol}{{\rm Vol}}
\newcommand{\bessel}{j_{d/2-1}}
\newcommand{\fbessel}{J_{d/2-1}}

\newcommand{\Bk}{\color{black}}

\def\qed{\hfill $\square$ \goodbreak}
\def\eps{\varepsilon}

\title{Stability in shape optimization with second variation}

\author{M. Dambrine - J. Lamboley}

\date{\today}
\begin{document}
\baselineskip14pt \maketitle

\begin{abstract}
We are interested in the question of stability in the field of shape optimization, with focus on the strategy using second order shape derivative. More precisely, we identify structural hypotheses on the hessian of the considered shape function, so that critical stable domains (i.e. such that the first order derivative vanishes and the second order one is positive) are local minima for smooth perturbations; as we are in an infinite dimensional framework, and that in most applications there is a norm-discrepancy phenomenon, this type of result require a lot of work. We show that these hypotheses are satisfied by classical functionals, involving the perimeter, the Dirichlet energy or the first Laplace-Dirichlet eigenvalue. We also explain how we can easily deal with constraints and/or invariance of the functionals. As an application, we retrieve or improve previous results from the existing literature, and provide new local stability results. We finally test the sharpness of our results by showing that the local minimality is in general not valid for non-smooth perturbations.
\end{abstract}

\noindent {\small {\it 2000MSC\ }: 49K20,  49Q10.}


\tableofcontents

\section{Introduction}

\subsection{Motivation and literature}

In this paper, we are interested in the question of stability in the field of shape optimization. More precisely, given $J:\A\to \R$ defined on $\A\subset\{\Om\textrm{ smooth enough open sets}\textrm{ in }\R^d\}$, we consider the optimization problem
\begin{equation}\label{eq:pb}
\min\left\{ J(\Om),\;\;\Om\in\A\right\},
\end{equation}
and  we ask the following question:

{\it if $\Om^*\in\A$ is a critical domain satisfying a stability condition (that is to say a strict second order optimality condition), can we conclude that $\Om^*$ is a strict local minimum for \eqref{eq:pb} in the sense that
\begin{equation}\label{eq:stabJ}
J(\Om)-J(\Om^*)\geq c d_{1}(\Om,\Om^*)^2, \;\;\;\;\textrm{ for every }\Om\in \V(\Om^*)
\end{equation}
where $c\in(0,\infty)$, $d_{1}$ is a distance among sets, and $\V(\Om^*)=\{\Om\in \A, d_{2}(\Om,\Om^*)<\eta\}$ is a neighborhood of $\Om^*$, relying on a (possibly different) distance $d_{2}$?}\\
(Note that the word distance is used here and in the rest of the paper as an intuitive notion, asserting that $\Om$ is far or close from the fixed shape $\Om^*$, and does not refer in general to the formal mathematical notion of distance).\\

\noindent{\bf Origin of the question:}

For example in \cite{Vogel}, the following terminology is used: given a function $f:X\to\R$ where $X$ is a Banach space, the property that a critical point $x\in X$ has a positive second order derivative is called {\it linear}-stability, and implies that $t\mapsto f(x+ty)$ has a local minimum at $t=0$ for every $y\in X$, while {\it nonlinear}-stability requires that $f(x)$ is less that $f(z)$ for any $z$ close to $x$. It is classical that, when dealing with infinitely dimensional parameters, these two notions do not coincide in general.

In the framework of shapes, this question has been raised in different settings, and its answer has sometimes been mistakenly considered as easily valid: for example, in the context of stable constant mean curvature surfaces, literature has focused for a while on giving sufficient conditions so that {\it linear}-stability would occur, without proving that it actually implied local minimality. This point was raised by Finn in \cite{Finn}, and some answers followed quickly, see \cite{Grosse-Brauckmann,Vogel,Vogel2}, though in the particular case of the ball and the isoperimetric problem, the  difficulty was already handled by Fuglede in \cite{Fuglede}. In the context of shape functionals involving PDE, the issue was raised by Descloux in \cite{Descloux} and a first solution was given in \cite{DambrinePierre,Dambrine}. All these examples will fit in the framework we describe in this paper.\\

\noindent{\bf Quantitative isoperimetric inequalities: different strategies}

During the last decade, starting with \cite{FMP}, this type of question gained interest in the community of isoperimetric inequalities and shape optimization, in particular three main methods were developed in a quite extensive literature, in order to get a stability result of the form \eqref{eq:stabJ} for the most classical problems \eqref{eq:pb}:

\begin{itemize}
\item Symmetrization technique,
\item Mass transportation approach,
\item Second order shape derivative approach.
\end{itemize}

As an example, we quote the $\sL^1$-stability result for the perimeter: there exists $c\in(0,\infty)$ such that
for every $V_{0}\in(0,\infty)$, 
\begin{equation}\label{eq:stabP}
\frac{P(\Om)-P(B)}{P(B)}\geq c d_{F}(\Om,B)^2, \;\;\;\;\textrm{ for every measurable set }\Om\textrm{ such that }|\Om|=V_{0},
\end{equation}
where $P$ denotes the perimeter (in the sense of geometric measure theory), $|\cdot|$ is the volume, $B$ is any ball of volume $V_{0}$, and 
$$d_{F}(\Om,B)=\inf_{\tau\in\R^d}\frac{|(\Om-\tau)\Delta B|}{|B|}$$
is known as the Fraenkel asymmetry (which can be seen as the $\sL^1$-`distance' to the ball, up to translations). For this specific example, all of these three strategies have been successfully applied, see \cite{FMP,FigMagPra,CiLe}.

Note in particular that the result is global: in that case a local result (in an $\sL^1$-neighborhood) implies a non-local one as it is shown in \cite[Lemma 5.1 and Lemma 2.3]{FMP}.

\subsection{Second order shape derivatives approach}

In this paper, {\bf we focus on the third strategy}, which recently received even more attention as in some examples, the other techniques could not be applied, or provided non-optimal results: as an example we quote the $\sL^1$-stability for the Faber-Krahn inequality, which was solved with symmetrization technique in \cite{FMP2}, but provided a higher (and less strong) exponent in \eqref{eq:stabJ}, and has been improved to an optimal exponent recently in \cite{BdPV} using this third strategy (see also \cite{FuscoZhang}).

One specific difficulty for this strategy is to define a framework of differential calculus within shapes. This can be done for example with the notion of shape derivatives, but one main drawback is that this is available only for reasonably smooth deformations of the initial shape, or in other words, for a rather strong distance $d_{2}$ (otherwise it is clear that classical functionals are not differentiable for non-smooth perturbations). Nevertheless, as it is shown for example in \cite{AFM}, the strategy can still lead to results for very weak distances (as the Fraenkel asymmetry), and can be decomposed in two main steps:\label{page:strategy}

\begin{itemize}
\item first, with the help of the differential setting and the fact that $\Om^*$ satisfies a strict second order optimality condition, prove a stability result for small and smooth perturbations of $\Om^*$; in other words, prove that \eqref{eq:stabJ} is valid where $d_{2}$ is a strong distance (and $d_{1}$ is limited by the properties of $J$, and is in general different from $d_{2}$, see below),
\item second, deduce from this first step that \eqref{eq:stabJ} is valid where $d_{1}=d_{2}$ is a weak distance (for example the Fraenkel asymmetry).
\end{itemize}
For the perimeter functional, the first step goes back to \cite{Fuglede}, and the second step is inspired by results in \cite{White,MorganRos}, though the complete result was achieved in \cite{CiLe}.
These two steps rely on very different arguments: in particular, the second step usually requires to adapt the regularity theory related to the optimization problem \eqref{eq:pb}, namely the notion of quasi-minimizer of the perimeter when the functional $J$ contains a perimeter term, or the regularity of free boundaries when $J$ involves an energy related to a PDE functional (see \cite{AFM,BdPV} respectively), so it strongly relies on specific properties of the functional $J$ under study. \\

{\it The aim of this paper is to describe a general framework so that the first step of the above strategy applies:
while this has been done in a few places in the literature, every time specifically for the functional that was under study, we aim at giving some general statements, and then show that these statements both apply to the examples already handled in the literature, and also to new examples. Despite getting a wider degree of generality, we also simplify many proofs and strategies found in the previous literature, as we describe in the rest of this introduction. We also show that the second step of the above strategy does not work with a similar degree of generality.\\
}

\noindent{\bf Neighborhood of shapes}

Before stating the main results, we briefly recall two classical ways to parametrize shapes in a neighborhood of a fixed one:
\begin{itemize}
\item {\bf Diffeomorphisms and shape derivatives:} we consider a shape to be a neighbor of $\Om$ if it is a deformation of $\Om$ by a diffeomorphism which is close to the identity.
More precisely,  $\Theta$ being  a Banach space such that $\sC^\infty(\R^d,\R^d)\subset\Theta\subset W^{1,\infty}(\mathbb{R}^d,\mathbb{R}^d)$, we consider shapes of the form $(\Id+\theta)(\Om)$ where $\|\theta\|_{\Theta}$ is small.


One defines the function $\mathcal{J}_{\Om}$ on a neighborhood of 0 in $\Theta$ by
$$\forall \theta\in\Theta, \;\;\;\mathcal{J}_{\Om}(\theta)=J[(\Id+\theta)(\Omega)].$$
One then uses {(in the whole paper)} the usual notion of {Fr\'echet-}differentiability: shape derivatives of $J$ at $\Om$ are the successive derivatives of $\mathcal{J}_{\Om}$ at $0$, when they exist. In particular, the first shape derivative is $J'(\Om):=\mathcal{J}_{\Om}'(0)$, a continuous linear form on $\Theta$ (the shape gradient), and the second order shape derivative is $J''(\Om):=\mathcal{J}_{\Om}''(0)$, a continuous symmetric bilinear form on $\Theta$ (the shape hessian).

\item{\bf Normal graphs:}\\
On the other hand, assuming that $\Om$ is $\sC^1$ (and $\n=\n_{\partial\Om}$ is its outer unit normal vector) we can consider ``normal graph'' on $\partial\Om$, that is $\Om_{h}$ 
such that
\begin{equation}
\label{formule:deformation:normale}
\partial\Om_{h}=\{x+h(x)\n(x), x\in \partial\Om\},
\end{equation}
where  $h\in X$, and $X$ is  a Banach space of scalar functions on $\partial\Om$.

\end{itemize}

Let us emphasize that even if the second method seems more restrictive, the two methods are equivalent in a neighborhood of $\Om$ (if $\Om$ is smooth enough) in the sense that one describes as many shapes with each methods (for suitable $\Theta$ and $X$): first, a normal graph $\Om_{h}$ is a deformation of $\Om$ for any $\boldsymbol{\xi}_{h}$ which is an extension to $\R^d$ of $h\n$ (and then $\j_{\Om}(h)=\mathcal{J}_{\Om}(\xi_{h})$), see the introduction of Section \ref{section:Chs}. Second,
if we consider diffeomorphims that are close to the identity, the boundaries of the perturbed domains are graphs over the boundary of the initial domain
, see for example Lemma 3.1 in \cite{NovruziPierre}.

However, it is not clear a priori that computing derivatives for normal graphs (derivatives of $\j_{\Om}:h\mapsto J(\Om_{h})$) is enough to describe shape derivatives (derivatives of $\mathcal{J}_{\Om}:\theta\mapsto J((\Id+\theta)(\Om))$): this issue is handled in our first result, see below. 


\subsection{Main results}

This paper contains three main results that we describe here:\\[-4mm]

\noindent{\bf 1. Structure of shape derivatives:} The tool of shape derivative using diffeomorphism as described just above is very convenient as most classical shape functionals are easily proven to be smooth in this setting (usually not using any regularity on the initial shape $\Om^*$, see more details in \cite[Chapter 5]{HenrotPierre} for example), and as noticed before, it is clear that computing derivatives in the sense of normal graphs is just a particular case (while the opposite seems not clear).
Nevertheless, one main drawback for our purpose is that as we are dealing with shapes, there is a lot of invariance for $\mathcal{J}_{\Om}$: any non-trivial diffeomorphisms that leaves $\Om$ invariant 
must lead to vanishing derivatives. It is therefore unreasonable to expect that the stability condition for optimal shapes writes $J''(\Om^*).(\xi,\xi) >0$ for $\xi\in\Theta\setminus\{0\}$.
For first order shape derivatives, this difficulty is well-known since Hadamard, who observed (in particular examples) that the shape gradient is a distribution supported on the boundary of the domain, acting only on the normal component of the deformation, 
 see for example \cite{DZ01Sha,HenrotPierre}, or \cite{LP} in a non-smooth setting. A similar observation can be made about second order shape derivatives, though the situation is more involved:
 it has been proven in \cite{NovruziPierre} that second order shape derivatives have a general structure, involving a quadratic form $\ell_{2}[J](\Om)$ acting only on normal components (as in the first order) and another term involving normal and tangential components, but relying on the first order derivative, see Theorem \ref{th:structure}. 
This structure is often observed in the literature on specific examples and after lengthy computations, while it can be used a priori to simplify the computations: indeed, once we know the shape functional is smooth (which can be shown without computations), this result implies that the computation of shape derivatives for purely normal deformations is sufficient to describe the second order derivative for any deformation (see also Remark \ref{rk:chain}). In particular, using the framework of normal graphs, we get $\j_{\Om}''(0)(h,h)=\ell_{2}[J](\Om)(h,h)$.\\[2mm]
{\it The first contribution of this paper is to give a new proof of the structure of second order shape derivatives, see Section  \ref{ssect:shapederivative}. Though the strategy in \cite{NovruziPierre} is quite natural as it shows that any small deformation of a shape can be seen (in a smooth way) as a normal deformation defined on the boundary, up to a change of parametrization of the boundary, we believe this new proof is less technical, and also quite natural as it only relies on the invariance properties mentioned before, and is therefore closer to the usual proof for the structure of first order shape derivatives. 
}
\\[2mm]
When $\Om^*$ is a critical domain, 
the proper stability assumption will reduce to the positivity of $\ell_{2}[J](\Om^*)$. Nevertheless, note that we will need in the proof of the stability result (Theorem \ref{th:main} below) to deal with second order shape derivatives at non-critical shapes as well. See also \cite{ACV16} for recent use of these structure results and application to numerical methods.\\

\noindent{\bf 2. Stability results: } Our second result provides an answer to the main goal of the paper. It gives the suitable assumptions on the functional so that linear stability implies non-linear stability. Before giving the statement, we describe these assumptions:

{\bf Assumption {\bf(C}$\!\!~_{\sH^{s_{2}}}${\bf)}:}\label{page:Chs} for $s_{2}\in(0,1]$, we say that the bilinear form $\ell$ acting on $\sC^\infty(\partial\Om)$ satisfies condition {\bf(C}$\!\!~_{\sH^{s_{2}}}${\bf)} (and by extension we say that $J$ satisfies the condition at $\Om$ if $\ell_{2}[J](\Om)$ does) if:
\begin{center}
\begin{minipage}{0.8\textwidth}
\begin{description}
\item[\hspace{-4mm}(C$\!\!~_{\sH^{s_{2}}}$)] there exists $s_{1}\in[0, s_{2})$ and $c_{1}>0$ such that $\ell=\ell_{m}+\ell_{r}$ with
\begin{equation*}
\left\{\begin{array}{l}\ell_{m}\textrm{ is lower semi-continuous in }\sH^{s_{2}}(\partial\Om) \\[3mm] \ell_{m}(\varphi,\varphi) \geq c_{1}|\varphi|_{\sH^{s_{2}}(\partial\Om)}^2, \ \forall \varphi\in \sC^\infty(\partial\Om),\\[3mm]
\ell_{r}\textrm{ continuous in }\sH^{s_{1}}(\partial\Om).
\end{array}\right.
\end{equation*}
\end{description}
\end{minipage}
\end{center}
where $|\cdot|_{\sH^{s_{2}}(\partial\Om)}$ denote the $\sH^{s_{2}}(\partial\Om)$ semi-norm. In that case, $\ell$ is naturally extended (by a density argument) to the space $\sH^{s_{2}}(\partial\Om)$.

{\bf Assumption {\bf(IT}$\!\!~_{\sH^{s},X}${\bf)}:}
given $\Om$, $s\in[0,1]$ and $X\subset W^{1,\infty}(\partial\Om)$ a Banach space, and assuming that $\j_{\Om}$ is $\sC^2$ in a neighborhood of 0 in  $X$, we say that $J$ satisfies  condition {\bf(IT}$\!\!~_{\sH^{s},X}${\bf)} (for ``improved Taylor'' expansion) at $\Om$  if:
\begin{center}
\begin{minipage}{0.85\textwidth}
\begin{description}
\item[\hspace{-6mm}(IT$\!\!~_{\sH^{s},X}$)] there exist $\eta>0$ and a modulus of continuity $\omega$ such that for every domain $\Om_{h}$ with $\|h\|_{X}\leq \eta$,
$$\left|J(\Om_{h})-J(\Om)-\ell_{1}(h)-\frac12\ell_{2}(h,h)\right|\leq \omega(\|h\|_{X})\|h\|_{\sH^{s}(\partial\Om)}^2,$$
\end{description}
\end{minipage}
\end{center}
where $\ell_{1}=\ell_{1}[J](\Om)=\j_{\Om}'(0),\; \ell_{2}=\ell_{2}[J](\Om)=\j_{\Om^*}''(0)$.

%

\begin{theorem}\label{th:main}
Let $\Om^*$ be a domain of class $\sC^{ 3}$, and $J$ a shape functional, twice Fr\'echet differentiable on a neighborhood of $\Om^*$ in $\sW^{1,\infty}$. We denote $\ell_{1}=\ell_{1}[J](\Om^*)$ and $\ell_{2}=\ell_{2}[J](\Om^*)$ (given by Theorem \ref{th:structure}), and assume that $J$ satisfies {\bf(C}$\!\!~_{\sH^{s_{2}}}${\bf)} and {\bf(IT}$\!\!~_{\sH^{s_{2}},X}${\bf)}  at $\Om^*$ for some $s_{2}\in(0,1]$ and $X$ a Banach space such that $\sC^\infty(\partial\Om^*)\subset X\subset \sW^{1,\infty}(\partial\Omega^*)$. \\
Then if $\Om^*$ is a critical and strictly stable shape for $J$, that is to say\footnote{here $\ell_{2}$ is a quadratic form, so $\ell_{2}>0$ on $X\setminus\{0\}$ means $\ell_{2}(\varphi,\varphi)>0$ for any $\varphi\in X\setminus\{0\}$}:
\begin{equation}
\label{eq:P''} 
\ell_{1}=0, \;\;\;\;\textrm{ and }\ell_{2}>0 \text{ on }\sH^{s_{2}}(\partial\Omega^*)\setminus\{0\},
\end{equation}
then 
there exist $\eta>0$ and $c=c(\eta)>0$ such that
\begin{equation}\label{eq:minstrict}
\forall \;\Om=\Om^*_{h}\textrm{ with } \|h\|_{X}\leq \eta, \;\;\;\;J(\Om)\geq J(\Om^*)+ c \|h\|_{\sH^{s_{2}}(\partial\Om^*)}^2.
\end{equation}
\end{theorem}
{\bf About the proof of Theorem \ref{th:main}, and hypotheses {\bf(C}$\!\!~_{\sH^{s_{2}}}${\bf)}} and {\bf(IT}$\!\!~_{\sH^{s},X}${\bf)}:
\begin{itemize}
\item First, as we deal with infinite dimensional differential calculus, we have to show that under assumption {\bf(C}$\!\!~_{\sH^{s_{2}}}${\bf)},
\begin{equation}\label{eq:coercivity}
\ell_{2}>0 \textrm{ on } \sH^{s_{2}}(\partial\Om^*)\setminus \{0\}\Leftrightarrow\Big( \exists \gamma>0,\;\;\;\forall \varphi\in \sH^{s_{2}}(\partial\Om^*), \;\;\; \ell_{2}(\varphi,\varphi) \geq \gamma\|\varphi\|_{\sH^{s_{2}}(\partial\Om^*)}^2\Big).
\end{equation}
The proof of this fact was inspired by \cite{Grosse-Brauckmann,AFM} dealing with particular cases (see also \cite[Theorem 4.11]{CT15Sec} for a similar example in a different context). 
Note also that the value of $s_{2}$ is determined by the shape functional $J$ (in practice $s_{2}$ usually does not depend on $\Om$): the choice of the distance $d_{1}$ in \eqref{eq:stabJ} is therefore limited by this coercivity property 
(see also \cite{Fuglede} where an upper bound of the isoperimetric deficit is given, in a smooth neighborhood). As we will notice in examples, when $J$ contains a perimeter term, $s_{2}=1$, while for PDE functionals we are dealing with here (see below when we describe examples), $s_{2}=1/2$. For an interesting result about the choice of $d_{1}$ in a non-smooth setting, see \cite{FJ} where the authors obtain an improved version of \eqref{eq:stabP} with $d_{1}$ being a stronger distance than the Fraenkel asymmetry (see also \cite{Neumayer} for the anisotropic case).
\item When writing the Taylor formula:
\begin{equation}
\label{formule:Taylor:Young}
J((\Id+\theta)(\Omega^*)) -J(\Omega^*)=  \frac{1}{2}\  \ell_{2}(\theta\cdot \n,\theta\cdot \n) + {\scriptscriptstyle \mathcal{O} }(\| { \theta}\|_{\Theta}^2).
\end{equation}
we have two issues:
\begin{itemize}
\item first, the remainder depends a priori on the full norm of $\theta$, while the second order term is only controlled with the norm of $\theta\cdot\n$,
\item the norm of differentiability $\Theta$ is in most cases stronger than the norm of coercivity given in the previous item, namely $\sH^{s_{2}}$,
\end{itemize}
so it is a priori not possible to control the sign of the term $\frac{1}{2}\  \ell_{2}(\theta\cdot \n,\theta\cdot \n) + {\scriptscriptstyle \mathcal{O} }(\| { \theta}\|_{\Theta}^2)$.
To solve the first issue, we show that we can restrict to normal perturbations
(this will turn out to be a significant simplification when dealing with constraints, see below). The second issue motivates assumption {\bf(IT}$\!\!~_{\sH^{s},X}${\bf)}. As far as we know, this phenomenon was first observed (in the context of shapes) when minimizing the perimeter as it is naturally differentiable in $\sW^{1,\infty}$ while the coercivity may only happen for the $\sH^1$-norm: see
 \cite{Fuglede} for the case of the 
 classical isoperimetric problem whose solution is the ball (see the next step to explain how we handle the translation invariance of the functional and the volume constraint) 
 and  \cite[Proof of Theorem 6]{Grosse-Brauckmann}, \cite[Equation 3.23]{BT}, \cite[Equation (1)]{White} in the more general framework of constant mean curvature surfaces
. See also \cite{Vogel,Vogel2} for similar observations with a different parametrization. 
Various geometric examples have been handled in the literature since these first examples, see \cite{DePM,FFMMM,BDF,Neumayer}. 
This difficulty is also well-known in the literature on second order optimality conditions in infinite dimension, especially applied to optimal control theory, see for example \cite{CT15Sec}. 
\item Note that Theorem \ref{th:main} has value only if one provides explainations on how to show assumptions {\bf(C}$\!\!~_{\sH^{s_{2}}}${\bf)} and {\bf(IT}$\!\!~_{\sH^{s},X}${\bf)} on concrete examples: Section \ref{ssect:exampleshapeder} and Section \ref{section:Chs} (and Theorem \ref{th:CHX} below) are dedicated to this issue.
\item {\bf Constraints and invariance:} in the isoperimetric problem, whose quantitative version is recalled in \eqref{eq:stabP}, we have to handle two extra difficulties: the functional is translation invariant, and  there is  a volume constraint in the optimization problem. Therefore one cannot expect \eqref{eq:P''} to be satisfied, see Section \ref{ssect:main:contrainte} for the suitable replacements.
 In \cite{DambrinePierre,Dambrine} the authors carefully handle the volume constraint by building a path preserving the volume and being almost normal, and prove that an estimate like {\bf(IT}$\!\!~_{\sH^{s_{2}},X}${\bf)} is valid for this more involved path. In \cite{AFM} a very similar approach is given, and they also handle the translation-invariance (which is not there in the example of \cite{DambrinePierre,Dambrine}) which implies a lot a technicalities.\\[2mm]
{\it 
We drastically simplify the presentation of \cite{DambrinePierre,Dambrine,AFM} by using an exact penalization method. More precisely (see {\bf Theorem \ref{th:main:contrainte}}), we prove that under the assumptions {\bf(C}$\!\!~_{\sH^{s_{2}}}${\bf)}
, the constrained optimality conditions (see Definition \eqref{def:stable}) implies the unconstrained conditions \eqref{eq:P''} when $J$ is replaced by
\begin{equation}\label{eq:penalization}
J_{\mu,C}=J-\mu \Vol+C\left(\Vol-V_{0}\right)^2+C\left\|\Bar-\Bar(\Om^*)\right\|^2,
\end{equation}
where $\mu$ is the Lagrange multiplier, $C\in(0,\infty)$ is large enough, and $\Bar$ is the barycenter functional. We then apply Theorem \ref{th:main} to $J_{\mu,C}$ 
which implies the constrained local minimality. It is clear, looking at the proofs of our results, that the situation we  describe in this paper is general and can be applied to other constraints or invariance.
}
\end{itemize}

\noindent{\bf 3. Condition {\bf(IT}$\!\!~_{\sH^{s},X}${\bf)} for $\lambda_{1}$ in Sobolev spaces:} 
As we mentioned before, the norm discrepancy issue has been dealt with for geometric functionals earlier in the literature: we briefly give proofs in Section \ref{sssect:geometric} for the sake of completeness.
In the specific context of shape optimization involving PDE, the issue was overcome in \cite{DambrinePierre,Dambrine}. More recently a very similar approach can be found in \cite{AFM}, see also Section \ref{ssect:literature}. The situation is much more involved than for geometric functionals, as it is much harder to write the remainder term in order to show condition {\bf(IT}$\!\!~_{\sH^{s},X}${\bf)} for suitable spaces.

In this case, it is more convenient to define a slightly different condition: given $\Om$, $s\in[0,1]$ and $X\subset W^{1,\infty}(\partial\Om)$ a Banach space, assuming that $\j_{\Om}$ is $\sC^2$ in a neighborhood of 0 in  $X$, we
say that $J$ satisfies condition {\bf(IC}$\!\!~_{\sH^{s},X}${\bf)} (for ``improved continuity'' in $h$) at $\Om$ if:
\begin{center}
\begin{minipage}{0.85\textwidth}
\begin{description}
\item[\hspace{-6mm}(IC$\!\!~_{\sH^{s},X}$)] there exist $\eta>0$ and a modulus of continuity $\omega$ such that for every domain $\Om_{h}$ with $\|h\|_{X}\leq \eta$, and all $t\in[0,1]$:
$$\left|j''(t)-j''(0)\right|\leq \omega(\|h\|_{X})\|h\|_{\sH^{s}}^2,$$
\end{description}
\end{minipage}
\end{center}
where $j:t\in[0,1]\mapsto J(\Omega_{t})$ for the path $(\Omega_{t})_{t \in [0,1]}$ connecting $\Om$ to $\Om_{h}$, and defined through its boundary
\begin{equation}
\label{definition:chemin:dans:formes}
\partial\Omega_{t} =\{x+\ th(x)\n(x), \ x\in \partial\Omega\}.
\end{equation}

Using the Taylor formula with integral remainder:
$$J(\Om_{h})-J(\Omega)=\ell_{1}(h)+\frac{1}{2}\ell_{2}(h,h)+\int_{0}^1 [j''(t)-j''(0)] (1-t) dt,$$
 it is easy to see that condition {\bf(IC}$\!\!~_{\sH^{s},X}${\bf)} implies {\bf(IT}$\!\!~_{\sH^{s},X}${\bf)}.\\

 We now recall the definition of two classical PDE functionals, $E$ the Dirichlet energy and $\lambda_{1}$  the first Dirichlet-eigenvalue:
\begin{equation}\label{eq:dirichlet}
E(\Om)=\min \left\{\frac{1}{2}\int_{\Om}|\nabla u|^2-\int_{\Om}u, \;\;u\in \sH^1_{0}(\Om)\right\}, \;\;\;\lambda_{1}(\Om)=\min \left\{\frac{\int_{\Om}|\nabla u|^2}{\int_{\Om}u^2}, \;\;u\in \sH^1_{0}(\Om)\right\}
\end{equation}

 In this paper, we prove the new following result:
\begin{theorem}\label{th:CHX}
Let $\Om$ be a bounded $\sC^3$ domain.
Then $E$ and $\lambda_{1}$ satisfy {\bf(IC}$\!\!~_{\sH^{1/2},\sW^{2,p}}${\bf)} for $p>d$.
\end{theorem}

This result is an improvement of the previous litterature in several ways: first in \cite{DambrinePierre,Dambrine} the authors prove {\bf(IC}$\!\!~_{\sH^{1/2},\sC^{2,\alpha}}${\bf)} for functionals similar to $E$, which is weaker. In \cite{AFM} the authors obtain a condition similar to {\bf(IC}$\!\!~_{\sH^{1/2},\sW^{2,p}}${\bf)} for $p>d$, but for a PDE functional which provides more regularity (see  \eqref{eq:AFMfunction}).
 Note that this improvement about spaces is not just a technical issue, as in \cite{AFM} the choice of $\sW^{2,p}$ rather than $\sC^{2,\alpha}$ is relevant for the second step of the strategy (described page \pageref{page:strategy}) when proving stability in an $\sL^1$-neighborhood (\cite[Section 4]{AFM}): indeed their regularization procedure needs to allow discontinuities of the mean curvature, see equation (4.9) in the proof of \cite[Theorem 4.3]{AFM}.
 Finally, as far as we know, the case of $\lambda_{1}$ was not known in the literature. 

\subsection{Old and new applications}

In order to justify the interest of our general statements, we provide several examples of functionals for which Theorems \ref{th:main} or \ref{th:main:contrainte} apply. We give here a short list, see Section \ref{sect:app} for more details.

\begin{itemize}
\item First, we retrieve with our results classical statements already existing in the literature: this relies on the computation of the first and second derivatives of the functionals, and the fact that they satisfy conditions {\bf(C}$\!\!~_{\sH^{s_{2}}}${\bf)} and {\bf(IT}$\!\!~_{\sH^{s_{2}},X}${\bf)} (for suitable $s_{2}$ and $X$). This includes the examples of \cite{DambrinePierre,Dambrine,AFM,BdPV}. We believe that despite the degree of generality of our approach, the proofs are less technical and more straightforward. 
\item Second, we 
apply our result to cases where only linear stability was studied: this includes the result 
in \cite{Nitsch} (see Proposition \ref{prop:applications}).
\item We finally provide new examples, which come with minor cost thanks to our results. One generic example we have in mind is the following: if $\Om^*=B$ is a ball of volume $V_{0}\in(0,\infty)$, 
then the conditions of Theorem \ref{th:main:contrainte} 
 are fulfilled for the functional $J=P+\gamma E$ ($P$ is the perimeter and $E$ is the Dirichlet energy) when $\gamma\geq \gamma_{0}$ and $\gamma_{0}\in(-\infty,0)$ (whose optimal value is given in Proposition \ref{prop:constante:optimale}), and we can conclude from our strategy that the ball is a local minimizer (in a $\sW^{2,p}$ neighborhood for $p>d$) of the following optimization problem
\begin{equation}\label{eq:P+gE}
\min\left\{ P(\Om)+\gamma E(\Om),\;\;|\Om|=V_{0}\right\}.
\end{equation}
{\it For $\gamma\geq 0$ this result is not surprising, since the ball minimizes both the perimeter and the Dirichlet energy. But this result is new and surprising when $\gamma$ is nonpositive
: there is  a competition between minimizing the perimeter and maximizing the Dirichlet energy. Another way to state the result is to say that
\begin{equation}\label{eq:quotient}
\frac{P(\Om)-P(B)}{E(\Om)-E(B)}\;\geq |\gamma_{0}|,\;\;\;\;\;\;\forall \Om\in\mathcal{V}(B),
\end{equation}
where $\mathcal{V}(B)=\{\Om=B_{h}, |\Om|=|B|\textrm{ and }\|h\|_{\sW^{2,p}}<\eta\}$, for some $\eta>0$.
}\\[2mm]
For a problem related to \eqref{eq:P+gE} when $\gamma<0$, see also \cite{GNR}. 
We also notice that local optimality of the ball is no longer valid when one consider a neighborhood of $\Om^*$ for a weak distance, for example the Frankel asymmetry, see Section \ref{ssect:non-stab}. Especially it means that the second step of the strategy described page \pageref{page:strategy} does not apply to \eqref{eq:P+gE} if $\gamma<0$, despite the fact that sets are minimizing the perimeter, and shows in what way the two steps of this strategy have different degree of generality.\\[2mm]
{\it In addition to this example, we obtain several new local isoperimetric inequalities, see {\bf Proposition \ref{prop:applications}} in Section \ref{subsection:applications1}.}

\end{itemize}

In Section \ref{sect:derivative}, we show a new proof of the Structure Theorem for second order shape derivatives. We also  
recall the classical examples of second order shape derivatives, noticing in particular in which norms they are continuous (which leads to the value of $s_{2}$ from assumption {\bf(C}$\!\!~_{\sH^{s_{2}}}${\bf)}), and focus on the case of the ball for which we diagonalize the shape hessians (which leads to the classical stability properties of the ball for these functionals). In Section \ref{sect:proofs} we state the version of Theorem \ref{th:main} adapted to the constrained/invariant case, we discuss the coercivity assumptions proving \eqref{eq:coercivity} (Lemma \ref{lem:coercivity}), and we prove Theorems \ref{th:main} and \ref{th:main:contrainte}. In Section \ref{section:Chs} we discuss assumption {\bf(IT}$\!\!~_{\sH^{s_{2}},X}${\bf)}, in particular we recall and improve existing results, and show Theorem \ref{th:CHX}. In Section \ref{sect:app}, we explain how our general results allow to retrieve known results, and then provide some new local isoperimetric inequalities, see Proposition \ref{prop:applications}. All these applications are simple corollaries of our main results, combined with the computations reminded in Section \ref{sect:derivative}. In the last Section, we show that similar results in non-smooth neighborhoods cannot be achieved with the same degree of generality.

\section{On Second order shape derivatives.}
\label{sect:derivative}

As for all the examples of this paper, we assume $J$ is a shape functional such that $\theta\in \Theta\mapsto J((\Id+\theta)(\Om))$ is of class $\sC^2$ in a neighborhood of $0$ in $\Theta=\sW^{1,\infty}(\R^d,\R^d)$. This simplifies the presentation, though similar proofs can be adapted to other functional spaces, see Remark \ref{rk:spaces}.

\subsection{Structure Theorem}\label{ssect:shapederivative}

It is well-known since Hadamard's work that the shape gradient is a distribution supported on the moving boundary and acting on the normal component of the deformation field. The second order shape derivative also has a specific structure as stated by A. Novruzi and M. Pierre in \cite{NovruziPierre}. We quote their result, and provide a new proof:

\begin{theorem}[Structure Theorem of first and second shape derivatives]
 \label{th:structure}
Let $\Theta=\sW^{1,\infty}(\R^d,\R^d)$, $\Om$ an open bounded domain of $\R^d$ and $J$ a real-valued shape function defined on $\mathcal{V}(\Om)=\{(\Id+\theta)(\Om), \|\theta\|_{\Theta}<1\}$. Let us define the function $\mathcal{J}_{\Om}$ on $\{\theta\in\Theta, \|\theta\|_{\Theta}<1\}$ by
$$\mathcal{J}_{\Om}(\theta)=J[(\Id+\theta)(\Omega)].$$
\begin{description}
\item[(i)] If $\mathcal{J}_{\Om}$ is differentiable at $0$ and $\Om$ is $\sC^2$, then there exists a  continuous linear form $\ell_{1}$  on $\sC^1(\partial\Omega)$ such that $\mathcal{J}_{\Om}'(0) \xi= \ell_{1}(\xi_{|\partial\Om}\cdot \n)$ for all $\xi \in \sC^{\infty}(\mathbb{R}^d,\mathbb{R}^d)$, where $\n$ denotes the unit exterior normal vector on $\partial\Om$.

\item[(ii)] If moreover $\mathcal{J}_{\Om}$ is twice differentiable at $0$ and $\Om$ is $\sC^3$, then there exists a  continuous symmetric bilinear form $\ell_{2}$  on $\sC^2(\partial\Omega) \times \sC^2(\partial\Omega)$ such that  for all $(\xi,\zeta) \in \sC^{\infty}(\mathbb{R}^d,\mathbb{R}^d)^2$
\begin{equation}\label{eq:structure2}
\mathcal{J}_{\Om}''(0) (\xi,\zeta)= \ell_{2}(\xi\cdot \n,\zeta\cdot \n) + \ell_{1}( \B(\zeta_{\tau}, \xi_{\tau})-\nablat (\zeta\cdot \n)\cdot\xi_{\tau}- \nablat (\xi\cdot \n)\cdot\zeta_{\tau} ),
\end{equation}
where $\nablat$ is the tangential gradient, $\xi_{\tau}$ and $\zeta_{\tau}$ stands for the tangential components of $\xi$ and $\zeta$, and $\B=D_{\tau}\n$ is the second fondamental form of $\partial\Om$.
\end{description}
\end{theorem}
With respect to this work, it is important to notice that at a critical domain for $J$, the shape hessian is reduced to $\ell_{2}$ and hence does not see the tangential components of the deformation fields. 

\begin{remark}
The requirement that $\Om$ is bounded is made only to simplify the presentation: the result remains valid replacing $\sC^1(\partial\Om)$ with $\sC^1_{c}(\partial\Om)$ and localizing the test functions.
\end{remark}

\begin{remark}\label{rk:reg}
As noticed in \cite[p. 225]{HenrotPierre}, with this degree of generality, the regularity assumption on $\Om$ are sharp. We could indeed wonder if $\ell_{1}$ can be extended as a continuous linear form on $\sC^0(\partial\Om)$; this is not true in general if $\Om$ is only assumed to be $\sC^1$, as the example of the perimeter shows (it would mean that the mean curvature is a Radon measure, which is not true for a $\sC^1$ domain).
Moreover, our strategy provides $\ell_{2}$ being continuous for the $\sC^2(\partial\Om)$-norm, while \cite{NovruziPierre} gives a better result with $\ell_{2}$ being continuous for the $\sC^1(\partial\Om)$-norm.
However, if we assume that $\ell_{1}$ can be extended as a continuous linear form on $\sC^0(\partial\Om)$, then point ${\bf (ii)}$ is valid assuming $\Om$ of class $\sC^2$ only, and $\ell_{2}$ is then continuous for the $\sC^1$-norm; it is easy to see how the proof adapts to this case, and we retrieve then an optimal result, see also \cite[Remark 2.8, Corollary 2.9]{NovruziPierre}, .
\end{remark}

\begin{remark}\label{rk:spaces}
Compare to the result in \cite{NovruziPierre}, we restricted ourself to the space $\Theta=\sW^{1,\infty}$ (or similarly $\sC^{1,\infty}:=\sW^{1,\infty}\cap\sC^1$, see the proof below), as all the functionals of this paper are differentiable in this space. Of course, the same proof can be adapted to spaces like $\sW^{k,\infty}$ for $k\geq 2$, which is important to handle higher order geometric or PDE functional.
\end{remark}
\begin{remark}
When $\xi=\zeta$, we get
$$\mathcal{J}_{\Om}''(0).(\xi,\xi)= \ell_{2}(\xi\cdot \n,\xi\cdot \n) + \ell_{1}(Z_{\xi}),\textrm{ where }Z_{\xi}= \B( \xi_{\tau}, \xi_{\tau})-2\nablat (\xi\cdot \n)\cdot\xi_{\tau}.$$
As noticed in \cite[Equation (7.5)]{AFM}, the term $Z_{\xi}$ can have be written in a different way:
$$Z_{\xi}=(\xi\cdot \n)\div(\xi)-\divt\left(\xi_{\tau}(\xi\cdot\n)\right)-H(\xi.\n)^2.$$
The advantage of $Z_{\xi}$ is that it clearly vanishes when $\xi_{\tau}=0$, but this second formulation can also have advantages, especially when $\xi$ has a vanishing divergence (as it is the case in \cite{AFM}) or when there are simplifications as it is the case for the volume (see Lemma \ref{lemme:expression:derivees} for the first equality):
\begin{equation}\label{eq:volume''}
\Vol''(\Om).(\xi,\xi)=\int_{\partial\Om} H(\xi\cdot\n)^2+\int_{\partial\Om}Z_{\xi} = \int_{\partial\Om}(\xi\cdot \n)\div(\xi).
\end{equation}
\end{remark}

 \begin{remark}\label{rk:chain}
It is sometimes considered that first and second order derivatives described in the previous theorem cannot handle the differentiation of $t\mapsto J(T_{t}(\Om))$ where $T\in \sC^2([0,\alpha[,\Theta)$ is not of the form $T_{t}=\Id+t\xi$. This is not true, as the chain rule formula easily gives (and is allowed when we have proven the Fr\'echet-differentiability of the functionals, which is valid for all the functionals of this paper):
$$\frac{d^2}{dt^2}J(T_{t}(\Om))=\mathcal{J}_{\Om}''(T_{t}-\Id).\left(\frac{d}{dt}T_{t},\frac{d}{dt}T_{t}\right)+\mathcal{J}_{\Om}'(T_{t}-\Id).\left(\frac{d^2}{dt^2}T_{t}\right)$$
and the structure result can then be applied. For example, if $T_{t}$ is the flow of the vector field $\xi$ as it is usually done in the speed method, we obtain:
$$\frac{d^2}{dt^2}J(T_{t}(\Om))_{|t=0}=\mathcal{J}_{\Om}''(0).\left(\xi,\xi\right)+\mathcal{J}_{\Om}'(0).\left((D\xi)\cdot\xi\right).$$
Another interesting case is that if $\Om$ is a critical shape for $J$, namely $\mathcal{J}_{\Om}'(0)\equiv 0$, and if $T_{t}=\Id+t\xi+\frac{t^2}{2}\eta+o(t^2)$ where $o(t^2)$ has to be understood with the norm $\|\cdot\|_{\Theta}$, then we always have
$$\frac{d^2}{dt^2}J(T_{t}(\Om))_{|t=0}=\ell_{2} (\xi\cdot\n,\xi\cdot\n).$$
\end{remark}

\noindent{\bf Proof of Theorem \ref{th:structure}}
We only focus on the second order derivative, as the first order one is classical (see for example \cite{HenrotPierre,DZ01Sha,LP}). For $k\in \N$, we define $\sC^{k,\infty}:=\sC^k\cap \sW^{k,\infty}(\R^d,\R^d)$ equipped with the same norm as $\sW^{k,\infty}$, which is also a Banach space and is more adapted to approximation by smooth functions. 
Let $\xi,\zeta\in \sC^{\infty}$  compactly supported, and denote $\gamma, \delta$ their respective flow, namely
$$\left\{\begin{array}{ccl}\frac{d}{dt}\gamma_{t}(x)&=&\xi(\gamma_{t}(x))\\\gamma_{0}(x)&=&x\end{array}\right.\;\;\;\;\;\;
\left\{\begin{array}{ccl}\frac{d}{dt}\delta_{t}(x)&=&\zeta(\delta_{t}(x))\\\delta_{0}(x)&=&x\end{array}\right.$$
Thanks to our assumption on $\xi$, we easily check that the function $T\in\Theta\mapsto \xi\circ(T+\Id)\in\Theta$ is locally Lipschitz and $\sC^2$, and therefore these ODE admits solutions defined on $(-t_{0},t_{0})$ and such that $[t\mapsto\gamma_{t}-\Id ,t\mapsto \delta_{t}-\Id]$ are in $\sC^2((-t_{0},t_{0}),\Theta)$. As a consequence, $(t,s)\mapsto\gamma_{s}\circ\delta_{t}-\Id\in \Theta$ is well-defined in a neighborhood of $(0,0)$ and $\sC^2$. 

Let now assume that $\zeta\cdot\n=0$. Then from classical criterion of invariance of sets with the flow, we have $\delta_t(\Om)=\Om$ for every $t$ small enough, so $J(\gamma_s\circ\delta_t(\Om))=\J_\Om(\gamma_s\circ\delta_t-\Id)$ is independent of $t$. Differentiating successively with respect to $t$ and $s$ at $(0,0)$, we obtain:
   $$\J_\Om''(0).(\xi,\zeta)+\J_\Om'(0).(D\xi\cdot\zeta)=0,\;\;\forall\xi\in\sC^\infty_{c},\;\;\forall \zeta\in K\cap \sC^\infty_{c},$$
 where $K=\Ker(\Phi)$ and $\Phi:\xi\in\Theta\mapsto\xi_{|\partial\Om}\cdot\n$.

We define $b : (\xi,\zeta)\in\sC^{2,\infty}\times\sC^{1,\infty} \mapsto \J_\Om''(0).(\xi,\zeta)+\J_\Om'(0).(D\xi\cdot\zeta)$ which is a continuous bilinear functional that vanishes for $\zeta\in K$, for any fixed $\xi$. Therefore we can write, using quotient properties,
$b(\xi,\zeta)=\widetilde{b}(\xi,\zeta_{|\partial\Om}\cdot\n)$ where
$\widetilde{b} : \sC^{2,\infty}\times \sC^1(\partial\Om) \rightarrow \R$ is continuous (a priori we only get that $\widetilde{b}$ is separately continuous but with Banach-Steinhaus Theorem, it implies continuity), as $\Phi$ induces an isomorphism between $\Theta/K$ and $\Phi(\Theta)=\sC^1(\partial\Om)$ equipped with the $\sC^1$ norm (using that $\Om$ is of class $\sC^2$).
Moreover by construction we have:
$$\J_\Om''(0).(\xi,\zeta)+\J_\Om'(0).(D\xi\cdot\zeta)=\tilde{b}(\xi,\zeta_{|\partial\Om}\cdot\n), \;\;
\forall\; \xi,\zeta\in\sC^{2,\infty}\times\sC^{1,\infty}.$$
Using the symmetry of $\J_\Om''(0)$, we can write, for every $(\xi,\zeta)\in \sC^{2,\infty}$:
\begin{equation}\label{eq:sym}
\tilde{b}(\zeta,\xi_{|\partial\Om}\cdot\n)-\tilde{b}(\xi,\zeta_{|\partial\Om}\cdot\n)=\J_\Om'(0).(D\zeta\cdot\xi-D\xi\cdot\zeta)
\end{equation}
Our goal is now to apply this formula to $\zeta_{\n}$ the normal component of $\zeta$, which needs to be extended as a vector field on $\R^d$. To that end, we introduce $P_{\partial\Om}$ the projection on $\partial\Om$, which is well-defined and $\sC^2$ in a neighborhood of $\partial\Om$, as $\Om$ is assumed to be $\sC^3$ (see for example \cite{DZ01Sha}). Then if $\varphi$ is defined on $\partial\Om$, we set $\widetilde{\varphi}(x)=\varphi(P_{\partial\Om}x)\chi(x)$ where $\chi$ is a smooth function with $\chi=1$ in a neighborhood of $\partial\Om$, and $\chi=0$ outside a compact set (in other words, $\varphi$ is extended so that it is constant in the normal direction). This operator $\varphi\mapsto\widetilde{\varphi}$ is continuous from $\sC^2(\partial\Om)$ to $\sC^{2,\infty}$.
Let us define then $\zeta_{\n}:=\widetilde{(\zeta\cdot\n)\n}$ the extension of the normal component of $\zeta$.
Defining
the bilinear form $\ell_0(\varphi_1,\varphi_2)=\tilde{b}(\widetilde{\varphi_1\n},\varphi_2)$, continuous on 
$\sC^2(\partial\Om)\times\sC^1(\partial\Om)$ (and a priori non symmetric), we obtain
$$\begin{array}{rcl}
\J_\Om''(0).(\xi,\zeta)&=&\tilde{b}(\xi,\zeta\cdot\n)-\J_\Om'(0).(D\xi\cdot\zeta)\\
&=&\tilde{b}(\zeta_{\n},\xi\cdot\n)-
\J_\Om'(0).(D\zeta_{\n}\cdot\xi-D\xi\cdot\zeta_{\n})-\J_\Om'(0).(D\xi\cdot\zeta)\:\:\:\:\:\textrm{(using \eqref{eq:sym})}\\
&=&\ell_0(\zeta\cdot\n,\xi\cdot\n)-\J_\Om'(0).(D\zeta_{\n}\cdot\xi - D\xi\cdot\zeta_{\n}+ D\xi\cdot\zeta)\\
&=&\ell_0(\zeta\cdot\n,\xi\cdot\n)-\J_\Om'(0).(D\zeta_{\n}\cdot\xi+ D\xi\cdot\zeta_\tau)
\end{array}
$$
where $\zeta_\tau=\zeta-\zeta_{\n}$. We now use $D\zeta_{\n}=D_\tau\zeta_{\n}$, because thanks to our choice of extension operator, $\zeta_{\n}$ is constant in the direction $\n$ (by definition, $D_\tau a=Da-(Da\cdot\n)\n$), and therefore $D\zeta_{\n}\cdot\xi=D_\tau\zeta_{\n}\cdot\xi_\tau$. Moreover, $D\xi\cdot\zeta_\tau=D_\tau\xi\cdot\zeta_\tau$.

Using a symmetrization of the previous formula, we obtain
$$\begin{array}{rcl}
\J_\Om''(0).(\xi,\zeta)&=&\frac{1}{2}\Big[\ell_0(\zeta\cdot\n,\xi\cdot\n)+\ell_0(\xi\cdot\n,\zeta\cdot\n)-\J_\Om'(0).(D_\tau\zeta_{\n}\cdot\xi_\tau+D_\tau\xi\cdot\zeta_\tau+D_\tau\xi_{\n}\cdot\zeta_\tau+D_\tau\zeta\cdot\xi_\tau)\Big]\\[3mm]
&=&\ell_2(\xi\cdot\n,\zeta\cdot\n)-\frac{1}{2}\J_\Om'(0)\cdot\Big(2D_\tau\zeta\cdot\xi_\tau+2D_\tau\xi\cdot\zeta_\tau-D_\tau\xi_\tau\cdot\zeta_\tau-D_\tau\zeta_\tau\cdot\xi_\tau\Big)
\end{array}
$$
where we defined $\ell_2(\xi\cdot\n,\zeta\cdot\n)=\frac{1}{2}(\ell_0(\zeta\cdot\n,\xi\cdot\n)+\ell_0(\xi\cdot\n,\zeta\cdot\n))$, which is a continuous bilinear form on $\sC^2(\partial\Om)^2$.

From the structure of the first order derivative, and using the formula
$$\transposee{D_\tau}\xi_\tau\cdot\n+\transposee D_\tau\n\cdot\xi_\tau=0$$
(obtained by tangentially differentiating $\xi_\tau\cdot\n=0$), we finally obtain (using the $\C^3$ regularity of $\partial\Om$ so that $D_{\tau}\n$ belongs to the space of definition of $\ell_{1}$)
$$\begin{array}{rcl}
\J_{\Om}''(0).(\xi,\zeta)&=&\ell_2(\xi\cdot\n,\zeta\cdot\n)-\frac{1}{2}\ell_1\Big((2D_\tau\zeta\cdot\xi_\tau+2D_\tau\xi\cdot\zeta_\tau)\cdot\n-\zeta_\tau\cdot(\transposee D_\tau\xi_\tau\cdot\n)-\xi_\tau\cdot(\transposee D_\tau\zeta_\tau\cdot\n)\Big)\\[3mm]
&=&\ell_2(\xi\cdot\n,\zeta\cdot\n)+\ell_1\Big((D_{\tau}\n \cdot\zeta_{\tau})\cdot \xi_{\tau}-\nablat (\zeta\cdot \n)\cdot\xi_{\tau}- \nablat (\xi\cdot \n)\cdot\zeta_{\tau} \Big)
\end{array}
$$
(where we used that $D_{\tau}\n$ is symmetric), which concludes the proof (a priori, $\ell_0$ depends on the extension operator that has been chosen, but as in the final formula the extension only appears in $\ell_2$ which does not depend of the extension operator).
\qed

\subsection{Examples of shapes derivatives}\label{ssect:exampleshapeder}

For an open bounded (smooth enough) set $\Omega \subset \R^d$, we consider in this section (and in the rest of the paper) its volume $|\Omega|$, its perimeter $P(\Omega)=\H^{d-1}(\partial\Om)$, its Dirichlet energy $E(\Omega)$ and its first eigenvalue of the Dirichlet Laplace operator $\lambda_{1}(\Om)$ (see \eqref{eq:dirichlet}). The existence and computations of the shape derivatives of these functionals are well known, see for example \cite[Chapter 5]{HenrotPierre}.
 We denote the mean curvature (understood as the sum of the principal curvatures of $\partial\Omega$)  by $H$, $\B=D_{\tau} \n$ is the second fundamental form of $\partial\Omega$, and $\|\B\|^2$ 
 is the sum of the squares of the principal curvatures of $\partial\Omega$.
\begin{lemma}[Expression of shape derivatives]
\label{lemme:expression:derivees}
If $\Omega$ is $\sC^2$, one has, for any $\varphi\in \sC^\infty(\partial\Om)$,
\begin{itemize}
\item $\displaystyle{\ell_{1}[\Vol](\Omega).\varphi=\int_{\partial\Omega}\varphi,\hspace{2cm}
 \ell_{2}[\Vol](\Omega).(\varphi,\varphi)=\int_{\partial\Omega} H \varphi^2.}$\\
\item
$\displaystyle{\ell_{1}[P](\Omega).\varphi=\int_{\partial\Omega}H \varphi,\hspace{1.9cm}
\ell_{2}[P](\Omega).(\varphi,\varphi)=\int_{\partial\Omega} |\nablat \varphi|^2+\int_{\partial\Omega} \left[H^2-{\|\B\|^2}
\right] \varphi^2}$\\
\item
$\displaystyle{\ell_{1}[E](\Omega).\varphi= - \cfrac{1}{2}\int_{\partial\Omega} (\partial_{n}u)^2 \varphi, \hspace{0.7cm}
 \ell_{2}[E](\Omega).(\varphi,\varphi)={\langle  \partial_{n} u \ \varphi,\Lambda(\partial_{n}u \ \varphi) \rangle_{\sH^{1/2}\times\sH^{-1/2}}} +\int_{\partial\Omega}\left[\partial_{n} u+ \frac{1}{2} H (\partial_{n}u)^2\right] \varphi^2}$\\[2mm]
where  $u\in\sH^1_{0}(\Omega)$ is the unique solution to $-\Delta u=1$, $\Lambda:\ \sH^{1/2}(\partial \Omega) \rightarrow  \sH^{-1/2}(\partial \Omega)$ is the Dirichlet-to-Neumann map defined as $\Lambda(\psi) = \partial_{n} \He(\psi)$ where $\He$ is the harmonic extension operator from $\sH^{1/2}(\partial\Omega)$ into $\sH^{1}(\Omega)$:
\begin{equation}
\label{definition:V(h)}
-\Delta \He(\psi)=0 \text{ in }\Omega, \ \ \ \He(\psi) =\psi  \text{ on }\partial\Omega,
\end{equation}
\item
$\displaystyle{\ell_{1}[\lambda_{1}](\Omega).\varphi= - \int_{\partial\Omega} (\partial_{n}v)^2 \varphi, \hspace{0.9cm}
\ell_{2}[\lambda_{1}](\Omega).(\varphi,\varphi)= \int_{\partial\Omega}  2 w(\varphi)\ \partial_{n} w(\varphi)+H (\partial_{n}v)^2 \varphi^2}$\\[2mm]
where $v$ is the normalized eigenfunction (solution in $\sH^1_{0}(\Omega)$ of $-\Delta v=\lambda_{1} v$ with $v\geq0$ in $\Omega$ and $\|v\|_{\sL^2(\Omega)}=1$) and  $w(\varphi)$ is the solution of
\begin{equation}
\label{definition:w}
\left\{
\begin{array}{rcl}
 -\Delta w(\varphi)&=&\lambda_{1} w(\varphi)  - v  \displaystyle\int_{\partial\Omega} (\partial_{n}v)^2 \varphi \text{ in }\Omega, \\[10pt]
w(\varphi)& =&-\varphi  \partial_{n} v\text{ on }\partial\Omega,\\[5pt]
\displaystyle\int_{\Omega} v \ w(\varphi)&=&0.
\end{array}
\right.	
\end{equation}
\end{itemize}

\end{lemma}

A fundamental fact for this work appears here in the expression of the shape hessians. {Even if they are  derived for regular perturbations,} they are naturally defined and continuous on different Sobolev spaces on $\partial\Omega$:
\begin{lemma}[Continuity of shape Hessians]
\label{continuite:derivee:seconde}
If $\Omega$ is $\sC^2$, there is a constant $C>0$ such that 
$$ |\ell_{2}[P](\Omega).(\varphi,\varphi)|\leq C \|\varphi\|^2_{\sH^1(\partial\Omega)}, \,\;\;\;\; |\ell_{2}[\Vol](\Omega).(\varphi,\varphi)|\leq C \|\varphi\|^2_{\sL^2(\partial\Omega)},$$
$$ |\ell_{2}[E](\Omega).(\varphi,\varphi)|\leq C \|\varphi\|^2_{\sH^{1/2}(\partial\Omega),} \;\;\;\;\; |\ell_{2}[\lambda_{1}](\Omega).(\varphi,\varphi)|\leq C \|\varphi\|^2_{\sH^{1/2}(\partial\Omega)}. $$
\end{lemma}
Therefore, from this Lemma, it is natural to consider the extension of these bilinear forms to their space of continuity. 

\subsection{The case of  balls}\label{ssect:balls}

In this section, we describe the shape derivatives of the previous functionals when the set $\Om$ is a ball. This will be very efficient when studying if one can apply Theorems \ref{th:main} and \ref{th:main:contrainte} to the ball, see Section \ref{sect:app}.

Let us focus on the ball $B_{1}$ of radius $1$. For the Dirichlet energy $E$, we remark that  $u(x)=(1-|x|^2)/2d$  solves $-\Delta u=1$ in $\sH^1_{0}(B_{1})$ and satisfies $\partial_{n} u=-\frac{1}{d}$ on $\partial B_{1}$. For $\lambda_{1}$, we recall that the eigenvalue and eigenfunction are 
$$\lambda_{1}(B_{1})=\bessel^2 \text{ associated to } v(x)\ =\ \alpha_{d} \ |x|^{1-d/2} \ \fbessel\left({	\bessel \ |x|}\right),$$
where $\bessel$ is the first zero of Bessel's function $\fbessel$ and $\alpha_{d}$ a normalization constant.  
Moreover, from \cite[p. 35]{Henry}, the eigenfunction satisfies 
\begin{equation}
\label{definition:gammad}
\partial_{n} v =\sqrt {\cfrac{2}{P(B_{1})}} \ \bessel:=\beta_{d}, \text{ so that }\beta_{d}^2=\frac{2\lambda_{1}(B_{1})}{P(B_{1})}.
\end{equation}
We obtain the {\bf shape gradients}:
\begin{eqnarray*}
\ell_{1}[\Vol](B_{1}).\varphi    =\int_{\partial B_{1}}\varphi, \label{boule:gradient:volume}&&
\ell_{1}[P](B_{1}).\varphi        = (d-1)\int_{\partial B_{1}}\varphi,  \\
\ell_{1}[E](B_{1}).\varphi        = - \cfrac{1}{2d^2} \int_{\partial B_{1}}  \varphi, &&
\ell_{1}[\lambda_{1}](B_{1}).\varphi  =\ -\  \beta_{d}^2 \int_{\partial B_{1}}  \varphi. 
\end{eqnarray*}
Let us notice that these four shape gradients at balls are colinear. As a consequence, the balls are critical domains for the perimeter, the  Dirichlet energy and $\lambda_{1}$ (or any sum of these functionals) under a volume constraint, and these formula easily provide the value of the Lagrange-multiplier.

{  Let us turn our attention to the {\bf hessians}. The value of $\ell_{2}[\lambda_{1}]$ is a bit more involved, so we deal with it in the next lemma. For the other functionals, it is known from Lemma \ref{lemme:expression:derivees} that:}
\begin{subequations}
\begin{align*}
\ell_{2}[\Vol](B_{1}).(\varphi,\varphi)	&	= (d-1)\int_{\partial B_{1}}  \varphi^2, \\[2mm]
\ell_{2}[P](B_{1}).(\varphi,\varphi)		&	=  \int_{\partial B_{1}} |\nablat \varphi|^2+(d-1)(d-2)\int_{\partial B_{1}} \varphi^2,\\[2mm]
\ell_{2}[E](B_{1}).(\varphi,\varphi)		&= \cfrac{1}{d^2} \langle \varphi ,\Lambda \varphi \rangle_{\sH^{1/2}\times \sH^{-1/2} } - \frac{d+1}{2d^2}  \int_{\partial B_{1}} \varphi^2. 
\end{align*}
\end{subequations}

In order to see that the quadratic forms associated to the Lagrangian are coercive on their natural spaces, it is useful to study the diagonalized form of these Hessians. To that end, we use spherical harmonics defined as the restriction to the unit sphere of harmonic polynomials. 
We recall here facts from \cite[pages 139-141]{SteinWeiss}. We let $\mathcal{H}_{k}$ denote the space of spherical harmonics of degree $k$ (that is, the restriction to $\partial B_{1}$ of homogeneous polynomials in $\R^d$, of degree $k$). It is also the eigenspace of the Laplace-Beltrami operator on the unit sphere associated with the eigenvalue $-k(k+d-2)$. Let $(Y^{k,l})_{1\leq l \leq d_{k}}$ be an orthonormal basis of $\mathcal{H}_{k}$ with respect to the $\sL^2(\partial B_{1})$ scalar product. The 
family $(Y^{k,l})_{k\in \mathbb{N}, 1\leq l \leq d_{k}}$ is a Hilbert basis of $\sL^2(\partial B_{1})$. Hence, any function $\varphi$ in ${  \sL^{2}(\partial B_{1})}$ can be decomposed:
\begin{equation}
\label{decomposition:harmoniques:spheriques}
\varphi(x)= \sum_{k=0}^{\infty} \sum_{l=1}^{d_{k}} \alpha_{k,l}(\varphi) Y^{k,l}(x), \ \ \text{ for }|x|=1.
\end{equation}
Then, by construction, the function $h$ defined by
\begin{equation}\label{eq:h}
h(x) = \sum_{k=0}^{\infty} |x|^k  \sum_{l=1}^{d_{k}} \alpha_{k,l}(\varphi) Y^{k,l}\left(\frac{x}{|x|}\right)
, \ \ \text{ for }|x|\leq 1,
\end{equation}
is harmonic in $B_{1}$ and satisfies $h=\varphi$ on $\partial B_1$. 
\textcolor{black}{Moreover, the sequence of coefficients $\alpha_{k,l}$ characterizes the Sobolev regularity of $\varphi$: indeed $\varphi\in \sH^{s}(\partial B_{1})$ if and only if the sum $ \sum_{k} (1+k^2)^{s}\sum_{l} |\alpha_{k,l}|^2$ converges.}
We can now state the following lemma expressing the previous shape hessians are diagonal on this basis. 

\begin{lemma} \label{diagonalisation:hessiennes}
Using the decomposition \eqref{decomposition:harmoniques:spheriques}, we have ($\beta_{d}$ is the constant defined in \eqref{definition:gammad})
\begin{align*}
\ell_{2}[\Vol](B_{1}).(\varphi,\varphi) & =   \sum_{k=0}^{\infty}\sum_{l=1}^{d_{k}} (d-1) \ \alpha_{k,l}(\varphi)^2 ,\\
\ell_{2}[E](B_{1}).(\varphi,\varphi) & =   \sum_{k=0}^{\infty}\sum_{l=1}^{d_{k}} \left[ \frac{1}{d^2} \ k \ - \frac{d+1}{2d^2}   \right] \alpha_{k,l}(\varphi)^2 ,\\
\ell_{2}[P](B_{1}).(\varphi,\varphi)&= \sum_{k=0}^{\infty}\sum_{l=1}^{d_{k}} \left[k^2+(d-2)k+(d-1)(d-2) \right] \ \alpha_{k,l}(\varphi)^2 ,\\
\ell_{2}[\lambda_{1}](B_{1})(\varphi,\varphi)&= \beta_{d}^2 \left( 3  \alpha^2_{0,1}(\varphi) + \sum_{k=1 }^\infty   \sum_{l=1}^{d_{k}} 2\left[ k+\frac{d-1}{2}- \bessel \cfrac{J_{k+d/2}(\bessel)}{J_{k-1+d/2}(\bessel)}  \right]   \alpha_{k,l}^2(\varphi) \right).
\end{align*}
\end{lemma}

\begin{proof} 
First we check that
$$\int_{\partial B_{1}}\varphi^2= \sum_{k=0}^{\infty}\sum_{l=1}^{d_{k}} \alpha_{k,l}(\varphi)^2, \;\;\;\;\;\int_{\partial B_{1}}|\nablat \varphi|^2=-\int_{\partial B_{1}}\varphi \ \Deltat \varphi =\sum_{k=0}^{\infty}k(k+d-2)\sum_{l=1}^{d_{k}} \alpha_{k,l}(\varphi)^2.$$
Then, we precise the term involving the Dirichlet-to-Neumann map that appears in the shape hessian of the Dirichlet energy. Using $h$ defined in \eqref{eq:h} 
and Green formula, we have:
\begin{eqnarray*}
\langle \varphi ,\Lambda \varphi \rangle_{ \sH^{1/2}\times \sH^{-1/2} }&=&\int_{\partial B_{1}} \varphi \partial_{n}h =\int_{B_{1}} |\nabla h|^2\\
 &= &\int_{0}^1 \left( \int_{\partial B_{r}} \left( (\partial_{n} h)^2 + |\nablat h|^2\right) d\sigma \right)  dr = \int_{0}^1\left( \int_{\partial B_{r}} \left( (\partial_{n} h)^2 - h \Deltat h\right) d\sigma \right) dr\\
& =&   \sum_{k=0}^{\infty}\sum_{l=1}^{d_{k}}  \int_{0}^1  r^{d-1} \left[k ^2 r^{2(k-1)}  +  \frac{k(k+d-2)}{r^2}\ r^{2k}\right] dr\  \alpha_{k,l}(\varphi)^2 \\
&= &   \sum_{k=0}^{\infty}\sum_{l=1}^{d_{k}} \left[\frac{k^2}{2k+d-2}  +  \frac{k(k+d-2)}{2k+d-2} \right]  \alpha_{k,l}(\varphi)^2=    \sum_{k=0}^{\infty}\sum_{l=1}^{d_{k}} k\  \alpha_{k,l}(\varphi)^2.
\end{eqnarray*}
We obtain $\ell_{2}[\Vol],\  \ell_{2}[P]$ and $\ell_{2}[E]$ by gathering these elementary terms. 

Let us now consider the case of the first eigenvalue. 
We apply \cite[p 35]{Henry} (see also \cite{Rayleigh} and \cite{Shimakura}): for a second order volume preserving path, that is $t\mapsto T_{t}$ such that $|T_{t}(\Om)|=|\Om|+o(t^2)$ for small $t$, we have
$$ \left(\frac{d^2}{dt^2}\lambda_{1}(T_{t}(B_{1}))\right)_{|t=0}=\sum_{k=1 }^\infty   \sum_{l=1}^{d_{k}} 2\beta_{d}^2\ \left[ k+d-1- \bessel \cfrac{J_{k+d/2}(\bessel)}{J_{k-1+d/2}(\bessel)}  \right]   \alpha_{k,l}^2(\varphi)$$
where $\varphi=(\frac{d}{dt}T_{t})_{|t=0}\cdot\n$ and we have used the recursive formula for Bessel function $J'_{\nu}(z)=(\nu/z)J_{\nu}(z)-J_{\nu+1}(z)$ to adapt his expression to our notations (\cite[section 9.1.27, p 361]{AbramowitzStegun}). To deduce $\ell_{2}[\lambda_{1}]$ from this computation, we introduce $\theta$ a smooth vector field which is normal on $\partial B_{1}$ and denote $\varphi=\theta\cdot\n$. We assume that $\int_{\partial B_{1}}\varphi=\alpha_{0,1}(\varphi)=0$. It is then clear that there exists $\xi$ such that $T_{t}:=\Id+t\theta+\frac{t^2}{2} \xi$ is volume preserving at the second order, that is to say 
$$\ell_{2}[\Vol](B_{1})(\varphi,\varphi)+\ell_{1}[\Vol](B_{1})(\psi)=0,$$
where $\psi=\xi\cdot\n$. Then we observe that for a smooth shape functional $J$ and for such $t\mapsto T_{t}$,
$$\left(\frac{d^2}{dt^2}J(T_{t}(B_{1}))\right)_{|t=0}=\ell_{2}[J](B_{1})(\varphi,\varphi)+\ell_{1}[J](B_{1})(\psi),$$
and therefore, denoting $\mu$ the Lagrange multiplier such that $\ell_{1}[\lambda_{1}-\mu\Vol](B_{1})=0$, we obtain
\begin{eqnarray*}
\left(\frac{d^2}{dt^2}\lambda_{1}(T_{t}(B_{1}))\right)_{|t=0}&=&\ell_{2}[\lambda_{1}](B_{1})(\varphi,\varphi)+\ell_{1}[\lambda_{1}](B_{1})(\psi)=\ell_{2}[\lambda_{1}](B_{1})(\varphi,\varphi)+\mu\ell_{1}[\Vol](B_{1})(\psi)
\\&=&\ell_{2}[\lambda_{1}](B_{1})(\varphi,\varphi)-\mu\ell_{2}[\Vol](B_{1})(\varphi,\varphi)
\end{eqnarray*}

Then, we get, as here $\mu=-\beta_{d}^2$:
\begin{align*}
\ell_{2}[\lambda_{1}](\varphi,\varphi)&=\left(\frac{d^2}{dt^2}\lambda_{1}(T_{t}(B_{1}))\right)_{|t=0}+\mu\ell_{2}[\Vol](B_{1})(\varphi,\varphi),\\
&=\sum_{k=1 }^\infty   \sum_{l=1}^{d_{k}} 2\beta_{d}^2\ \left[ k+d-1- \bessel \cfrac{J_{k+d/2}(\bessel)}{J_{k-1+d/2}(\bessel)}  \right]   \alpha_{k,l}^2(\varphi) -\beta_{d}^2\sum_{k=0 }^\infty   \sum_{l=1}^{d_{k}}(d-1)a_{k,l}^2(\varphi),\\
&=\sum_{k=1 }^\infty   \sum_{l=1}^{d_{k}}2 \beta_{d}^2\ \left[ k+\frac{d-1}{2}- \bessel \cfrac{J_{k+d/2}(\bessel)}{J_{k-1+d/2}(\bessel)}  \right]   \alpha_{k,l}^2(\varphi).
\end{align*}
It remains to compute the coefficient associated to the mode $k=0$. It suffices to  consider the deformations as $T_{t}(x)=x+tx$ mapping the ball $B_{1}$ onto the ball $B_{1+t}$. Here $\varphi=1$ and $\alpha_{0,1}(\varphi)=P(B_{1})^{1/2}$. Since $\lambda_{1}$ is homogeneous of degree $-2$, we get $\lambda(t):=\lambda_{1}(T_{t}(B_{1}))=(1+t)^{-2} \lambda_{1}(B_{1})$  so that $\lambda''(0)= 6 \lambda_{1}(B_{1})=6 \frac{\lambda_{1}(B_{1})}{P(B_{1})}\alpha_{0,1}(\varphi)^2$.
\end{proof}

\section{Stability Theorems}
\label{sect:proofs}

\subsection{Statement with volume constraint and translation invariance}\label{ssect:main:contrainte}
In the introduction, we gave the unconstrained version of the stability result, see Theorem \ref{th:main}. As in many applications we need to deal with a volume constraint and a translation invariance of the functional, we describe here the corresponding statement (a similar method can be applied to other kinds of invariance or constraints).

\begin{definition}\label{def:stable}
Let $\Om^*$ be a shape of class $\sC^3$ and $J$ a shape functional defined and twice shape differentiable in a neighborhood of $\Om^*$ in $\sW^{1,\infty}$.
\begin{itemize}
\item We say that $\Om^*$ is a critical domain for $J$ under volume constraint if
\begin{equation}\label{eq:critical}
\forall \varphi\in { \sC^\infty(\partial\Om^*)}\textrm{ such that } {  \ell_{1}[\Vol]}(\Om^*).\varphi=\int_{\partial\Om^*}\varphi=0, \;\;\;\;\;\;\;\;{  \ell_{1}[J](\Om^*)}.(\varphi)=0.
\end{equation}
It is equivalent to the existence of $\mu\in\R$ such that
$({  \ell_{1}[J]-\mu \ell_{1}[\Vol])}(\Om^*)=0$, $\mu$ is called a Lagrange multiplier.
\item When $\Om^*$ is a critical domain for $J$ under volume constraint, we say that $\Om^*$ is a strictly stable shape for $J$ under volume constraint and up to translations if 
\begin{equation}\label{eq:stable}
\forall \varphi\in T(\partial\Om^*)\setminus\{0\},\;({  \ell_{2}[J]-\mu \ell_{2}[\Vol])}(\Om^*).(\varphi,\varphi) >0
\end{equation}
where 
\begin{equation}\label{eq:T}
T(\partial\Om^*):=\left\{\varphi\in \sH^s(\partial\Om^*), \;\; \int_{\partial\Om^*}\varphi=0, \;\;\int_{\partial\Om^*}\varphi\overrightarrow{x}=\overrightarrow{0}\right\},
\end{equation}
$\mu$ is the Lagrange multiplier associated and $s\geq 0$ is the lowest index so that $\ell_{2}[J](\Om^*)$ is continuous on $\sH^s(\partial\Om^*)$ (see Lemma \ref{continuite:derivee:seconde}).
\end{itemize}
\end{definition}

\begin{theorem}\label{th:main:contrainte}
Let $\Om^*$ of class $\sC^{3}$, 
and $J$ a shape functional, translation invariant and twice Fr\'echet differentiable on a neighborhood of $\Om^*$ in $\sW^{1,\infty}$. We assume:
\begin{itemize}
\item {\bf Structural hypotheses:} there exists $s_{2}\in(0,1]$ and $X$ a Banach space with $\sC^\infty(\R^d)\subset X\subset \sW^{1,\infty}(\R^d)$
such that $J$ satisfies 
{\bf(C}$\!\!~_{\sH^{s_{2}}}${\bf)} and {\bf(IT}$\!\!~_{\sH^{s_{2}},X}${\bf)}  at $\Om^*$,
\item {\bf Necessary optimality conditions:}
\begin{itemize}
\item $\Om^*$ is a critical shape under volume constraint for $J$,
\item $\Om^*$ is a strictly stable shape for $J$ under volume constraint and up to translations.
\end{itemize}
\end{itemize}
Then 
there exists $\eta>0$ and $c=c(\eta)>0$ such that: 
$$\forall \;\Om=\Om^*_{h}\textrm{ such that }\|h\|_{X}\leq \eta\textrm{ and }|\Om|=|\Om^*|, \;\;\;\;J(\Om)\geq J(\Om^*) + c d_{X}(\Om,\Om^*)^2,$$
where 
\begin{equation}\label{eq:distance}
d_{X}(\Om,\Om^*)=\inf\{\|g\|_{\sH^{s_{2}}(\partial\Om^*)}, g\textrm{ such that }\exists \tau\in\R^d, \;\;\Om+\tau=\Om^*_{g}\}
\end{equation}
\end{theorem}

\subsection{About coercivity and condition  {\bf(C}$\!\!~_{\sH^{s}}${\bf)}}\label{ssect:coercivity}

Usually the coercivity property for the second order derivative (of the functional or of the Lagrangian) has to be proven by hand on each specific example by studying the lower bound of the spectrum of the bilinear form $\ell_{2}$ defined in Theorem \ref{th:structure}, typically thanks to Lemma \ref{diagonalisation:hessiennes}. Nevertheless, when $\ell_{2}$ enjoys some structural property, coercivity can be more easily checked as a consequence of the following general lemma (in.

\begin{lemma}
\label{lem:coercivity}
Let $M$ be the boundary of a Lipschitz-domain in $\R^d$, $s_{2}\in[0,1]$, and  
$V$ a linear subspace of $\sH^{s_{2}}(M)$, closed for the weak convergence in  $\sH^{s_{2}}(M)$. If $\ell$,  a quadratic form defined on $\sH^{s_{2}}(M)$ satisfies condition {\bf(C}$\!\!~_{\sH^{s_{2}}}${\bf)} (see page \pageref{page:Chs}), 
then the following propositions are equivalent:
\begin{center}
\begin{minipage}{0.7\textwidth}
\begin{description}
\item[(i)] $\ell(\varphi,\varphi)>0$ for any $\varphi \in V\setminus\{0\}$.
\item[(ii)] $\exists \gamma>0,\;\;\; \ell(\varphi,\varphi) \geq \gamma\|\varphi\|_{\sH^{s_{2}}(M)}^2$ for any $\varphi \in V$.
\end{description}
\end{minipage}
\end{center}
\end{lemma}
\begin{remark}
In practice, we apply this lemma 
 when $V$ is either $\sH^{s_{2}}(\partial\Om)$ or $T(\partial\Om)$ defined in \eqref{eq:T}.
\end{remark}
\begin{proof}
The implication 
{\bf(ii)} $\implies$ {\bf (i)} is trivial. 
Assume {\bf (i)} and let $(\varphi_k)_{k}$ a minimizing sequence for the problem 
$$\inf \left\{ \ell(\varphi,\varphi),\ \varphi \in V, \|\varphi\|_{\sH^{s_{2}}}=1 \right\}.$$
Up to a subsequence, $\varphi_k$ weakly converges in  $\sH^{s_{2}}(M)$ to some $\varphi_{\infty} \in V$. By the compactness of the embedding of $\sH^{s_{2}}(M)$ into $\sH^{s_{1}}(M)$, $\varphi_k \rightarrow \varphi_{\infty}$ in $\sH^{s_{1}}(M)$  so that $\ell_{r}(\varphi_k,\varphi_k) \rightarrow \ell_{r}(\varphi_\infty,\varphi_\infty)$.
We distinguish two cases: if $\varphi_{\infty}\ne 0$, $\liminf_{k} \ell_{m}(\varphi_{k},\varphi_{k}) \geq \ell_{m} (\varphi_{\infty},\varphi_{\infty})$ by the lower semi continuity of $\ell_{m}$, so that
$\liminf_{k} \ell(\varphi_{k},\varphi_{k})\geq \ell(\varphi_{\infty},\varphi_{\infty})>0$ by assumption {\bf(i)}.
Now, if $\varphi_{\infty}=0$, then  as the norm $\|\cdot\|_{\sH^{s_{2}}}$ is equivalent to the norm $\|\cdot\|_{\sH^{s_{1}}}+|\cdot|_{\sH^{s_{2}}}$, we know that $|\varphi_{k}|_{\sH^{s_{2}}}$ is bounded from below by a positive constant, and using {\bf(C}$\!\!~_{\sH^{s_{2}}}${\bf)}, 
$$\liminf_{k} \ell(\varphi_{k},\varphi_{k})=\liminf_{k} \ell_{m}(\varphi_{k},\varphi_{k}) \geq c_{1} \liminf_{k} |\varphi_{k}|_{\sH^{s_{2}}}^2 >0.$$ 

\end{proof}
\begin{remark} 
The equivalence between coercivity in $\sL^2$ and $\sH^1$ was already known in the context of stable minimal surface it appears in the work \cite{Grosse-Brauckmann} of Grosse-Brauckmann. In \cite{AFM}, the previous lemma is proven in the particular case of the functional under study (see also Section \ref{ssect:literature}).
\end{remark}

\begin{remark}
When one applies this lemma to a shape hessian, assumption (i) may seem unnatural. Indeed, shape derivatives are usually defined for regular perturbations that are dense subsets of $\sH^{s}(\partial\Omega)$ and one could expect to assume only $\ell(\varphi,\varphi)>0$ for $\varphi \in V\setminus\{0\}$ smooth enough. But this assumption may not be sufficient: indeed the function $\varphi_{\infty}$ in the proof above may not be smooth and therefore not admissible to test the positivity property. Therefore, the shape hessian $\ell$ has to be first extended by continuity to the whole $\sH^{s}(\partial\Omega)$ \textcolor{black}{(see assumption \eqref{eq:P''} in Theorem \ref{th:main} and \eqref{eq:stable} for Theorem \ref{th:main:contrainte}), see Lemma \ref{lemme:expression:derivees} for such an extension in classical examples. However in some cases, we may expect regularity for $\varphi_{\infty}$, see for example \cite[Remark 1]{DePM}.}
\end{remark}

We conclude  this section noticing that the shape hessians of the model functionals from Section \ref{sect:derivative} satisfies {\bf(C}$\!\!~_{\sH^{s_{2}}}${\bf)}: 
\begin{itemize}
\item The perimeter satisfies {\bf(C}$\!\!~_{\sH^{1}}${\bf)} with 
$$ \ell_{m}[P](\Omega)(\varphi,\varphi)=\int_{\partial\Omega} |\nablat \varphi|^2 \;\;\;\;\text{ and }\;\;\;\; \ell_{r}[P](\Omega)(\varphi,\varphi)=\int_{\partial\Omega} \left[H^2-\|\B\|^2
\right] \varphi^2 \;\;\;\;\;\textrm{ (here we can choose }s_{1}=0).$$
\item The Dirichlet energy and $\lambda_{1}$ satisfy {\bf(C}$\!\!~_{\sH^{1/2}}${\bf)} (again $s_{1}=0$):
\begin{eqnarray*} &&\ell_{m}[E](\Omega)(\varphi,\varphi)= \langle \partial_{n} u \varphi,\Lambda(\partial_{n}u \varphi) \rangle_{\sH^{1/2}\times\sH^{-1/2}} \;\;\;\;\text{ and }\;\;\;\;\ell_{r}[E](\Omega)(\varphi,\varphi)=\int_{\partial\Omega}\left[\partial_{n} u+ \frac{1}{2} H (\partial_{n}u)^2\right] \varphi^2,\\
&&\ell_{m}[\lambda_{1}](\Omega).(\varphi,\varphi)= \int_{\partial\Omega}  2 w(\varphi)\ \partial_{n} w(\varphi)\;\;\;\; \text{ and }\;\;\;\;\ell_{r}[\lambda_{1}](\Omega)(\varphi,\varphi)= \int_{\partial\Omega} H (\partial_{n}v)^2 \varphi^2.
\end{eqnarray*}
\end{itemize}
 
\begin{remark}
Let us emphasize that condition {\bf(C}$\!\!~_{\sH^{s}}${\bf)} may not be valid in some interesting examples.   
 Shape functionals used for domain reconstruction from boundary measurements provide in general non-coercive Hessians. With the examples treated in  \cite{AfraitesDambrineKateb}, \cite{BadraCaubetDambrine} one can find critical shape whose hessian is positive but is not coercive (for any $\sH^s$-norm). 

More precisely, for a reconstruction function $J$ related to this kind of inverse problem (for example the least square fitting to data), the Riesz operator corresponding to the shape Hessian $\ell_{2}[J]$ at a critical domain is compact. This means, that one cannot expect an estimate of the kind $J(\Omega_{t})-J(\Om_{0})\geq c t^2$ with a constant $c$ uniform in the deformation direction. This explains also why regularization is required in the numerical treatment of this type of problem. This fact is well-known in the inverse problem community. 

There are also situations where the objective is flat up to fourth order (see \cite{DambrineKateb}). 
\end{remark}
\subsection{Proof of Theorem \ref{th:main}}\label{ssect:proofcontrainte}

Let $\Omega^*$ be a domain satisfying the assumption of Theorem \ref{th:main}. Let $\eta>0$ and let $\Omega=\Om^*_{h}$ with $\|h\|_{X}<\eta$. 
Then from {\bf(IT}$\!\!~_{\sH^{s},X}${\bf)} we have
\begin{equation*}
J(\Omega)-J(\Omega^*)=\underbrace{\ell_{1}[J](\Om^*)(h)}_{=0}+\frac12\ell_{2}[J](\Om^*)(h,h)+\omega(\|h\|_{X})\|h\|^2_{\sH^s}
\end{equation*}			
Using {\bf(C}$\!\!~_{\sH^{s_{2}}}${\bf)}, we can apply Lemma \ref{lem:coercivity} and there is a constant $\gamma>0$  such that
\begin{equation*}
\ell_{2}[J](\Om^*).(h,h)\geq \gamma \|h\|_{\sH^{s_{2}}}^2.
\end{equation*}
Therefore there exists $\eta$ small enough such that if $\|h\|_{X}\leq \eta$, then $\omega(\|h\|_{X})\leq \frac\gamma4$ and then
$$J(\Omega)-J(\Omega^*)\geq \cfrac{\gamma}{4}\|h\|_{\sH^{s_{2}}}^2.$$\qed

\subsection{Proof of Theorem \ref{th:main:contrainte}}

We denote $\mu$ the Lagrange multiplier associated to $J$. Therefore we consider $J_{\mu}=J-\mu\Vol$ and $\Om^*$ satisfies $J_{\mu}'(\Om^*)=0$.\\

\noindent{\bf Step 1: Stability under volume and barycenter constraint:} 
Under the structural hypotheses on $\ell_{2}[J](\Om^*)=\ell_{m}+\ell_{r}$ and the fact that $\ell_{2}[\Vol](\Om^*)$ is continuous in the $\sL^2$-norm,  we can applied Lemma \ref{lem:coercivity} to $\ell_{2}[J_{\mu}](\Om^*)$, so there are constants $c_{1},c_{2},c_{3}$ and $c_4>0$ such that
\begin{equation}\label{eq:continuity}
\forall \varphi\in  \sH^{s_{2}}(\partial\Om^*), \ \  |\ell_{m}(\varphi,\varphi)|\geq c_{1} |\varphi|_{\sH^{s_{1}}}^2\ \ \ \ |{ \ell_{r}}(\varphi,\varphi)|\leq c_{2}\|\varphi\|_{\sH^{s_{1}}}^2, \ \ \ \ |{  \ell_{2}}[\Vol](\Om^*).(\varphi,\varphi)|\leq c_{3}\|\varphi\|_{\sL^2}^2,
\end{equation}
\begin{equation}\label{eq:positivity}
\forall \varphi\in T(\partial\Om^*)
, \;\;\;\ell_{2}[J-\mu \Vol](\Om^*).(\varphi,\varphi)\geq {  c_{4}}\|\varphi\|_{\sH^{s_{2}}}^2.
\end{equation}

\noindent{\bf Step 2: Stability without constraint:} 
We consider 
$$J_{\mu,C}=J-\mu \Vol+C\left(\Vol-V_{0}\right)^2+C\left\|\Bar-\Bar(\Om^*)\right\|^2,$$
where $\Bar(\Om):=\int_{\Om}x$ and $\|\cdot\|$ is the euclidean norm in $\R^d$.
The shape $\Om^*$ still satisfies $J_{\mu,C}'(\Om^*)=0$. We claim that $\Om^*$ is a 
strictly stable shape for $J_{\mu,C}$ on the entire space $\sH^{s_{2}}(\partial\Om^*)$ when $C$ is big enough, that is to say
for all $\varphi$ in $\sH^{s_{2}}(\partial\Om^*)\setminus\{0\}$,
\begin{equation}\label{eq:positiveJC}
\ell_{2}[J_{\mu,C}](\Om^*).(\varphi,\varphi)>0.
\end{equation}
Indeed, if it was not the case, we would have the existence of $\varphi_{n}\in \sH^{s_{2}}(\partial\Om^*)\setminus\{0\}$ such that
\begin{equation}\label{eq:pos}
\ell_{2}[J_{\mu,n}](\Om^*).(\varphi_{n},\varphi_{n})\leq 0.
\end{equation}
Using \eqref{eq:continuity}, this leads to
\begin{equation}\label{eq:borne}
c_{1}|\varphi_{n}|_{\sH^{s_{2}}}^2-c_{2}\|\varphi_{n}\|_{\sH^{s_{1}}}^2-|\mu|c_{3}\|\varphi_{n}\|^2_{\sL^2}+2n\left(\int_{\partial\Om^*}\varphi_{n}\right)^2+2n\left\|\int_{\partial\Om^*}\varphi_{n} x\right\|^2\leq 0.
\end{equation}
Assuming by homogeneity that $\|\varphi_{n}\|_{\sH^{ s_{1}}}=1$ for every $n$, \eqref{eq:borne} implies that $(\varphi_{n})_{n}$ is bounded in $\sH^{s_{2}}$ and using the compactness of $\sH^{s_{2}}(\partial\Om^*)$ in $\sH^{s_{1}}(\partial\Om^*)$, we have up to a subsequence that $\varphi_{n}$  converges to $\varphi$ weakly  in $\sH^{s_{2}}$ and strongly in $\sH^{s_{1}}$. Therefore \eqref{eq:borne} implies first that $2n[\Vol'(\varphi_{n})^2+\Bar'(\varphi_{n})^2]$ is bounded, then 
 that $\varphi\in T(\partial\Omega^*)$ 
and then the semi-lower continuity assumption in {\bf(C}$\!\!~_{\sH^{s_{2}}}${\bf)} implies
$$\ell_{2}[J_{\mu}](\Om^*).(\varphi,\varphi)\leq 0,\;\;\;\textrm{ with }\|\varphi\|_{\sH^{s_{1}}}=1$$
which contradicts \eqref{eq:positivity}.\\

\noindent{\bf Step 3: Stability:} 
 It is now easy to see that $J_{\mu,C}$ satisfies both {\bf(C}$\!\!~_{\sH^{s_{2}}}${\bf)} and {\bf(IT}$\!\!~_{\sH^{s_{2}},X}${\bf)} at $\Om^*$ (using that $\Vol$ and $\Bar$ satisfy {\bf(IT}$\!\!~_{\sH^{0},\sW^{1,\infty}}${\bf)}, see Section \ref{sssect:geometric}), and for $C$ large enough we have \eqref{eq:positiveJC}, so applying Theorem \ref{th:main}, there exists $c>0$ and $\eta>0$ such that for every $\Omega=\Om_{h}$ with $ \|h\|_{X}<\eta$,  
\begin{equation*}
J_{\mu,C}(\Om)-J_{\mu,C}(\Om^*)\geq c \|h\|_{\sH^{s_{2}}}^2,
\end{equation*}
We then write this inequality in particular for shapes $\Om$ having the same volume and barycenter as $\Om^*$, 
and conclude the proof using the invariance of $J$ with translations.
\qed

\section{About Condition  {\bf(IT}$\!\!~_{\sH^{s},X}${\bf)}}
\label{section:Chs}

In this section, we show that our main examples satisfy condition {\bf(IT}$\!\!~_{\sH^{s},X}${\bf)} where $s$ is given in Section \ref{ssect:coercivity}, and $X$ is hoped to be as large as possible. Let us start with the notations we will use in this section.

Given $\Om$ an open set 
 and $h:\partial\Om\to \R$, we recall that $\Om_{h}$ is defined so that
$$\partial\Om_{h}=\{x+h(x)\n(x),x\in\partial\Om\}.$$

It will be useful to see $\Om_{h}$ as a deformation with a vector field. To that end, we assume $\Om$ of class $\sC^2$ so that the projection $\pi_{\partial\Om}$ on $\partial\Om$ is well-defined and $\sC^1$ in a neighborhood of $\partial\Om$, and we define
 $$h(x)=h(\pi_{\partial\Om}(x))\;\;\;\;\;\;\textrm{ and }\;\;\;\;\;\;\;\n(x)=\n(\pi_{\partial\Om}(x)),$$
in order to extend $h$ and $\n$ in a neighborhood of $\partial\Om$, and then we define $\boldsymbol{\xi}_{h} (x)= h(x)\n(x)$ in this neighborhood. With this construction, $\boldsymbol{\xi}_{h}$ is constant in the normal direction, so
$\div{\boldsymbol{\xi}_{h}}= \div(\n) h$. We can then extend it smoothly to $\R^d$, so that $\boldsymbol{\xi}_{h}\in\sW^{1,\infty}(\R^d,\R^d)$. 
Denoting $T_{h}=\Id+\boldsymbol{\xi}_{h}$, we have $\Om_{h}=T_{h}(\Om)$, and $\j_{\Om}(h)=\mathcal{J}_{\Om}({\boldsymbol{\xi}_{h}})=J(\Om_{h})$ (where $J$ is the shape functional under study). 
In this section, the notation $\widehat{w_{h}}$ stands for $w_{h}\circ T_{h}$ where $w_{h}$ is defined on $\Om_{h}$ or $\partial\Om_{h}$.

When studying condition {\bf(IC}$\!\!~_{\sH^{s},X}${\bf)} (which implies {\bf(IT}$\!\!~_{\sH^{s},X}${\bf)}), we focus on the path $\Om_{t}$ defined in \eqref{definition:chemin:dans:formes}, and we have $\Om_{t}=(\Id+t\boldsymbol{\xi}_{h})(\Om)$  and 
$j''(t)=\mathcal{J}_{\Om}''(t{\boldsymbol{\xi}_{h}}).(\boldsymbol{\xi}_{h},\boldsymbol{\xi}_{h})$ for all $t\in[0,1]$, where $j(t)=J(\Om_{t})$. Note that in this case we will notify the dependence of quantities with respect to $t$, but there is also a dependence in $h$ that we will not recall in order to simplify the notations : for example $\n_{h}$ will denote the exterior normal vector to $\Om_{h}$ and $\n_{t}$ the normal vector to $\Om_{t}$ while we should use $\n_{th}$.
Also, as we chose a vector field that is constant along the normal vector in a neighborhood of $\partial\Om$, we have (if $\|h\|_{\infty}$ is small enough)
\begin{equation}\label{eq:j''}
j''(t)=J''(\Om_{t})\cdot (\boldsymbol{\xi}_{h},\boldsymbol{\xi}_{h})=\mathcal{J}_{\Om_{t}}''(0)(\boldsymbol{\xi}_{h},\boldsymbol{\xi}_{h}).
\end{equation}

\subsection{Geometric quantities}\label{sssect:geometric}

\noindent{\bf $\bullet$ The volume:}

\begin{proposition}\label{prop:volume}
If $\Om$ is $\sC^2$, then $\Vol$ satisfies {\bf(IC}$\!\!~_{\sL^{2},\sW^{1,\infty}}${\bf)} at $\Om$.
\end{proposition}
\begin{remark}\label{rk:barycenter}
More generally (with a similar proof), we have that $\Om\mapsto \int_{\Om} f$ also satisfies {\bf(IC}$\!\!~_{\sL^{2},\sW^{1,\infty}}${\bf)} if $f\in C^1(\R^d)$. This is true in particular for the barycenter functional. 
\end{remark}

Before proving this result, we give a geometric Lemma, inspired by the results in \cite{Dambrine}. We recall that $J_{\partial\Om}(h):=\det{D T_{h}}|({}^t D T_{h}^{-1}){  \n} |$ is the surface jacobian, appearing when changing variables between $\partial\Om_{h}$ and $\partial\Om$. 
\begin{lemma}
\label{estimations:geometriques1}
We have the following Taylor expensions, where ${\mathcal{O} }$ denote a domination uniform in $x\in\partial\Om$,
\begin{itemize}
\item $J_{\partial\Om}(h)(x)=1+\ell_{1}^J(h(x),\nabla h(x))+\frac{1}{2}\ell_{2}^J(h(x),\nabla h(x))+\mathcal{O}\left(\|h\|_{\sW^{1,\infty}(\partial\Om)} \left(|h(x)|^2+|\nabla h(x)|^2\right)\right),$
\item $\widehat{\n}_{h}(x)=\n(x)+\ell_{1}^{\n}(h(x),\nabla h(x))+\frac{1}{2}\ell_{2}^{\n}(h(x),\nabla h(x))+\mathcal{O}\left(\|h\|_{\sW^{1,\infty}(\partial\Om)} \left(|h(x)|^2+|\nabla h(x)|^2\right)\right).$
\end{itemize}
where $(\ell_{1}^J,(\ell_{1}^{\n})_{i\in\llbracket1,d\rrbracket})$, $(\ell_{2}^J,(\ell_{2}^{\n})_{i\in\llbracket1,d\rrbracket})$ are respectively linear and quadratic forms on $\R^{d+1}$.
\end{lemma} 

\noindent{\bf Proof of Lemma \ref{estimations:geometriques1}:}
The first part follows simply from the fact that  $A\in M_{d}(\R)\mapsto\det(A)\left|({}^t A^{-1}){  \n} \right|$ is smooth in a neighborhood of $\Id$, and the fact that $D\boldsymbol{\xi}_{h}=h(D\n)+\nabla h\otimes\n$.\\
For the second part, we use a level-set parametrization: there exists $\phi$ of class $\sC^2$ such that $\Om=\{\phi<0\}$ and $\nabla\phi$ does not vanish in a neighborhood of $\partial\Om$, and then $\Om=\{\phi\circ T_{h}^{-1}<0\}$. Therefore
\begin{equation}
\label{variations:normale}
\widehat{\n}_{h}-\n=\frac{\nabla(\phi\circ T_{h}^{-1})}{|\nabla(\phi\circ T_{h}^{-1})|}\circ T_{h}-\frac{\nabla\phi}{|\nabla\phi|}=\frac{{}^t DT_{h}^{-1}.\nabla\phi}{|{}^t DT_{h}^{-1}.\nabla\phi|}-\frac{\nabla\phi}{|\nabla\phi|},
\end{equation}
and we conclude using the smoothness of $A\mapsto {}^t A^{-1}$ and $w\in\R^d\mapsto \frac{w}{|w|}$ in the neighborhood of $\Id$ and $\nabla\phi$ respectively.\qed

~\\
\noindent{\bf Proof of Proposition \ref{prop:volume}:}
We use \eqref{eq:volume''}, \eqref{eq:j''}, and the fact that $\div(\boldsymbol{\xi}_{h})=h\div(\n)$ (as $h$ is constant in the direction of $\n$). Therefore if $v(t)=\Vol(\Om_{t})$, we have:
$$v''(t)=\int_{\partial\Om_{t}}\boldsymbol{\xi}_{h}\cdot\n_{t}\div(\boldsymbol{\xi}_{h})=\int_{\partial\Om_{t}}\div(\n)(\n\cdot\n_{t}) h^2=\int_{\partial\Om}H (\n\cdot\widehat{\n}_{t}) h^2 J_{\partial\Om}(t).$$
 With Lemma \ref{estimations:geometriques1}, we easily obtain
$$|v''(t)-v''(0)|\leq Ct\|h\|_{\sW^{1,\infty}}\|h\|_{\sL^2}^2\leq C\|h\|_{\sW^{1,\infty}}\|h\|_{\sL^2}^2.$$
\qed

\begin{remark}\label{rk:volumebadproof}
We could try a direct proof estimating
$$|\Om_{h}|-|\Om|=\int_{\Om}\left(\det(\Id+D\boldsymbol{\xi}_{h})-1\right),$$
but a priori this only leads to the fact that the volume satisfies  {\bf(IT}$\!\!~_{\sH^{1},\sW^{1,\infty}}${\bf)}. 
In the spirit of \cite[Lemma 4.1]{Neumayer}, we could also try:
$$|\Om|=\frac{1}{d}\int_{\partial\Om}x\cdot\n_{h}=\frac{1}{d}\int_{\partial\Om}(x+h(x)\n(x))\cdot\widehat{\n}_{h} J_{\partial\Om}(h)
$$
but this leads to the same issue (see also Remark \ref{rk:comparison}).
\end{remark}

\noindent{\bf $\bullet$  The perimeter:}

\begin{proposition}
If $\Om$ is $\sC^2$, then $P$ satisfies {\bf(IT}$\!\!~_{\sH^{1},\sW^{1,\infty}}${\bf)} condition at $\Om$.
\end{proposition}

\begin{proof}
We follow exactly the second proof suggested in Remark \ref{rk:volumebadproof} and use Lemma \ref{estimations:geometriques1}:
\begin{equation*}\label{eq:taylorp}
P(\Om_{h})=\int_{\partial\Om_{h}}1=\int_{\partial\Om}J_{\partial\Om}(h)=
P(\Om)+\ell_{1}[P](\Om)(h)+\frac{1}{2}\ell_{2}[P](\Om)(h,h)+{\mathcal{O}(\|h\|_{\sW^{1,\infty}} }\|h\|_{\sH^1}^2).
\end{equation*} 
\end{proof}
\begin{remark}\label{rk:comparison}
It is interesting to compare the two strategies used for the volume and for the perimeter: indeed, for the volume we prefered to use condition {\bf(IC)}, while a similar strategy for the perimeter, as it is done in \cite{Dambrine} or in \cite[Proof of Theorem 3.9]{AFM} (but for a different path of shapes) lead to weaker results, namely {\bf(IC}$\!\!~_{\sH^{1},\sC^{2,\alpha}}${\bf)} and {\bf(IC}$\!\!~_{\sH^{1},\sW^{2,p}}${\bf)} respectively). 
\end{remark}

\subsection{PDE energies}

For PDE energies, a condition of the type {\bf(IC}$\!\!~_{\sH^{s},X}${\bf)} was studied first in \cite{DambrinePierre} where it is proven that in dimension two the Dirichlet energy satisfy {\bf(IC}$\!\!~_{\sH^{1/2},\sC^{2,\alpha}}${\bf)} (for a volume preserving path instead of a normal path), then a similar result is proven for general PDE functionals in any dimension in \cite{Dambrine}, either for the path \eqref{definition:chemin:dans:formes} or a volume preserving path.
More recently in \cite{AFM}, it was proven that the functional described in \eqref{eq:AFMfunction} involving the sum of the perimeter and a PDE functional (of a different kind than in \cite{Dambrine}) satisfies {\bf(IC}$\!\!~_{\sH^{1},\sW^{2,p}}${\bf)} for $p$ large enough, also for a volume preserving path, see also Section \ref{ssect:literature}. Finally, condition  {\bf(IC}$\!\!~_{\sH^{1/2},\sC^{2,\alpha}}${\bf)} is also established for the drag in a Stokes flow in \cite{CaubetDambrine}. 
 Thanks to our method to handle the volume constraint (see Section \ref{ssect:proofcontrainte}), we only need to deal with the normal path \eqref{definition:chemin:dans:formes}.

In this section, we prove Theorem \ref{th:CHX}, which includes the case of $\lambda_{1}$ (which seemed not to be handled in the literature), and we improve the result from \cite{Dambrine} by proving {\bf(IC}$\!\!~_{\sH^{s},X}${\bf)} with a smaller space $X$.
We give 4 preliminary steps to prove this result.  We only give the details for $\lambda_{1}$, as the case of $E$ is easier and the reader can follow \cite{Dambrine} or \cite[Appendix]{BdPV} and use the ideas below where we explain how to get $X$ to be $\sW^{2,p}$ instead of $\sC^{2,\alpha}$. We assume $\Om$ to be $\sC^3$.\\[-3mm]

\noindent{ { $\bullet$ Step 1: }\noindent{\bf{ Computing the second derivative along the path.}} 

Denoting $v_{t}$ the first normalized eigenfunction on $\Om_{t}$ and applying the structure Theorem to $\lambda_{1}$ (Lemma \ref{lemme:expression:derivees}) and \eqref{eq:j''}, we get
\begin{align}
\lambda_{1}''(\Om_{t}).(\boldsymbol{\xi}_{h},\boldsymbol{\xi}_{h})&= 2
\int_{\partial\Om_{t}}v_{t}'\partial_{\n_{t}}v_{t}'+\int_{\partial\Om_{t}}(\partial_{\n_{t}}v_{t})^2 \left[ H_{t}(\boldsymbol{\xi}_{h}\cdot \n_{t})^2-\B_{t}(({\boldsymbol{\xi}_{h})}_{\tau_{t}},{(\boldsymbol{\xi}_{h})}_{\tau_{t}})+2\nabla_{\tau_{t}}(\boldsymbol{\xi}_{h}\cdot \n_{t}){(\boldsymbol{\xi}_{h})}_{\tau_{t}}\right]\notag\\
&= 2
\underbrace{\int_{\partial\Om_{t}}v_{t}'\partial_{\n_{t}}v_{t}'}_{\mathcal{T}_{1}(t)}+\underbrace{\int_{\partial\Om_{t}}(\partial_{\n_{t}}v_{t})^2 \left[ H_{t}\alpha_{t}^2-\B_{t}(\beta_{t},\beta_{t})-2\nabla_{\tau_{t}}(\alpha_{t})\cdot\beta_{t}\right]h^2}_{\mathcal{T}_{2}(t)}\label{decompostion:derivee:seconde:directionnelle:a:t}\\
&\hspace{3cm}-2\underbrace{\int_{\partial\Om_{t}}(\partial_{\n_{t}}v_{t})^2 \alpha_{t}\left(\beta_{t}\cdot\nabla_{\tau_{t}}h\right)h}_{\mathcal{T}_{3}(t)}\;\;\;\;\;\textrm{ where }\alpha_{t}=\n_{t}\cdot \n, \;\;\;\;\;\;\;\beta_{t}=\alpha_{t}\n_{t}-\n.\notag
\end{align}

\noindent{ { $\bullet$ Step 2: }\noindent{\bf Geometric estimates:}} 

Similarly to Section \ref{sssect:geometric}, we denote $\widehat{w}_{h}=w_{h}\circ (\Id+\boldsymbol{\xi}_{h})$ where $w_{h}$ is defined on $\Om_{h}$ of $\partial\Om_{h}$. 
The following Lemma follows easily from Lemma \ref{estimations:geometriques1} (see \cite{Dambrine} for more details).
\begin{lemma}
\label{estimations:geometriques}
There is a constant $C$ depending on $\Omega$ such that for all $h$ in a neighborhood of 0 in $\sW^{2,p}(\partial\Om)$,
\begin{itemize}
\item $\|\widehat{J_{\partial\Om}(h)}-1\|_{\sL^\infty(\partial\Om)} \ \leq \ C \|h\|_{\sW^{1,\infty}(\partial\Om)},$ $\;\;\;\;\|\widehat{J_{\partial\Om}(h)}-1\|_{\sW^{1,p}(\partial\Om)} \ \leq \ C \|h\|_{\sW^{2,p}(\partial\Om)},$
\item $\|\widehat{H_{h}}-H\|_{\sL^p(\partial\Om)}\leq \ C \|h\|_{\sW^{2,p}(\partial\Om)},$ $\;\;\;\;\|\widehat{\B_{h}}-\B\|_{\sL^p(\partial\Om)}\leq \ C \|h\|_{\sW^{2,p}(\partial\Om)},$
\item $\|\widehat{\alpha_{h}}-1\|_{\sL^\infty(\partial\Om)}\leq C\|h\|_{\sW^{1,\infty}(\partial\Om)}$, \;\;\;\;$\|\widehat{\nabla_{\tau_{h}}\alpha_{h}}\|_{\sL^p(\partial\Om)}\leq C\|h\|_{\sW^{2,p}(\partial\Om)}$,
\item $\|\widehat{\beta_{h}}\|_{\sL^\infty(\partial\Om)}\leq C\|h\|_{\sW^{1,\infty}(\partial\Om)}$, $\;\;\;\;\|\widehat{\beta_{h}}\|_{\sW^{1,p}(\partial\Om)}\leq C\|h\|_{\sW^{2,p}(\partial\Om)}$.
\end{itemize}
\end{lemma} 

\noindent{ { $\bullet$ Step 3: }\noindent{\bf{Estimate of $\|\widehat{v_{\theta}}-v\|_{\sW^{2,p}}$}:}
This step is not specific to our chosen deformations $\boldsymbol{\xi}_{h}$ hence we present it for general deformations $\theta\in \sW^{1,\infty}(\R^d,\R^d)$, that is $v_{\theta}$ is the first Dirichlet eigenfunction on $(\Id+\theta)(\Om)$.

\begin{lemma}\label{regularite:vp:fp}
If $p>d$, the map $\theta\mapsto \widehat{v}_{\theta}$ from $\sW^{2,p}(\R^d,\R^d)$ with values in $\sW^{2,p}(\Omega)$ is $\sC^\infty$ around $0$. As a consequence, there is a neighborhood of $0$ in $\sW^{2,p}(\R^d,\R^d)$ and $C$ depending on $\Om$ only so that 
$$\| \widehat{v_{\theta}} -v_{0}\|_{\sW^{2,p}(\Om)} \leq C \| \theta-\Id \|_{\sW^{2,p}} .$$
\end{lemma}

\begin{proof}
We use the same strategy as  in \cite[Proof of Theorem 5.7.4]{HenrotPierre} and \cite{Henry} but with different functional spaces: precisely, we will apply the implicit function theorem to $\mathcal{F}: ~X\times Y  \times \R \to Z \times \R $ defined by
\begin{equation}\label{eq:implicit}
\mathcal{F}(\theta,v,\lambda)=\left( -\div{A(\theta)\nabla v} - \lambda J(\theta)v, \int_{\Omega} v^2J(\theta) -1\right)\;\;\;\;\;\textrm{where }\;\;\;\left\{\begin{array}{l}J(\theta)=\det(\Id+D\theta),\\[3mm]A(\theta)=J(\theta)(\Id+D\theta)^{-1}(\Id+{}^tD\theta)^{-1},\end{array}\right.
\end{equation}
for suitable spaces $X,Y,Z$. 
Using that $\sW^{1,p}$ is an algebra for $p>d$, we easily obtain that the maps $J$ and $A$ are $\sC^\infty$ around $0$ from $\sW^{2,p}(\R^d,\R^d)$ into $\sW^{1,p}(\R^d,\R^d)$. 
As a consequence, by Sobolev's embedding, the map $\mathcal{F}$ is $\sC^\infty$ around $(0,v_{0},\lambda_{0}:=\lambda_{1}(\Om))$ from $\sW^{2,p}(\R^d,\R^d) \times \sW^{2,p}(\Omega)\cap H^1_{0}(\Om)\times\R  $ into $\sL^{p} (\Omega)\times \R$. Besides 
$\mathcal{F}(0,v_{0},\lambda_{0})=(0,0)$ and the differential
$$\partial_{v,\lambda}\mathcal{F}(0,v_{0},\lambda_{0}).[w,\lambda]= \left( (-\Delta-\lambda_{0}) w-\lambda v_{0},2 \int_{\Omega} v_{0} w\right) $$ 
is an isomorphism from $\sW^{2,p} (\Omega)\cap H^1_{0}(\Om)\times \R$ into $\sL^{p} (\Omega)\times \R$ (see \cite[Lemma 5.7.3]{HenrotPierre} for details) and the conclusion follows.
\end{proof}

{\bf {$\bullet$ Step 4: }}\noindent{\bf estimation of the variation of the shape derivative of the eigenfunction:} 

The objective of this step is to prove the following estimate:
\begin{lemma}
\label{lemme:variation:derivee:fonction:propre}
There is $C,\eta$ depending only on $\Omega$ such that, if $\|h\|_{\sW^{2,p}(\partial\Om)}\leq \eta$, then
\begin{equation}
\label{eq::variation:derivee:fonction:propre}
\|\widehat{v'_{t}}-v'_{0}\|_{\sH^{1}(\Omega)}\leq C  \|h\|_{\sH^{1/2}(\partial\Omega)} \  \|h\|_{\sW^{2,p}(\partial\Om) }.
\end{equation}
\end{lemma}

This step is the most involved one when dealing with $\lambda_{1}$ instead of the Dirichlet Energy: the latter reduces in fact to the second step in the following proof. 

\begin{proof} 

We recall (see \eqref{definition:w}) that
\begin{equation}
\label{eq:v'}
\left\{
\begin{array}{rcl}
 -\Delta v_{t}'&=&\lambda_{1}(t) v'_{t} +\lambda_{1}'(t)v_{t} \text{ in }\Omega_{t}, \\[10pt]
v'_{t}& =&-(\partial_{\n_{t}}v_{t})\boldsymbol{\xi}_{h}\cdot \n_{t}\text{ on }\partial\Omega_{t},\\[5pt]
\displaystyle\int_{\Omega} v_{t}' v_{t}&=&0.
\end{array}
\right.	
\end{equation}

\emph{1. Splitting.}
We introduce $\He_{t}$ the harmonic extension on $\Om_{t}$ of $(\partial_{\n_{t}}v_{t})\boldsymbol{\xi}_{h}\cdot \n_{t}$.
Noticing that
$$\lambda_{1}'(t)=-\int_{\partial\Om_{t}}(\partial_{\n_{t}}v_{t})^2\boldsymbol{\xi}_{h}\cdot \n_{t}=\lambda_{1}(t)\langle v_{t},\He_{t}\rangle$$
where $\langle\cdot,\cdot\rangle$ is the scalar product in $L^2(\Om_{t})$, we decompose 
$$v_{t}'=-\pi_{t}\He_{t}+w_{t}$$
 where $\pi_{t}$ is the orthogonal projection on $E(t):=\{v_{t}\}^\perp$, and $w_{t}$ solves
 \begin{equation}
\label{reformulation:w:en:u}
\left\{
\begin{array}{rcl}
 (-\Delta-\lambda_{1}(t)) w_{t}&=&-\lambda_{1}(t) \pi_{t}\He_{t} \text{ in }\Omega_{t}, \\[10pt]
w_{t}& =&0\text{ on }\partial\Omega_{t},\\[5pt]
\displaystyle\int_{\Omega_{t}} v_{t} w_{t}&=&0.
\end{array}
\right.	
\end{equation}

We will now prove that each term of the splitting satisfies estimates like \eqref{eq::variation:derivee:fonction:propre}.

\emph{2. Estimate of the harmonic extension.}
Let us define $\mathcal{L}_{t}=\div(A_{t}\nabla\cdot)$ where $A_{t}=J_{t}.(\Id+tD\boldsymbol{\xi}_{h})^{-1}(\Id+t{}^tD\boldsymbol{\xi}_{h})^{-1}$ and  $J_{t}=\det(\Id+tD\boldsymbol{\xi}_{h})$, so that $\mathcal{L}_{t}\widehat{f_{t}}=\widehat{\Delta f_{t}}$ if $f_{t}$ is defined on $\Om_{t}$. Then as $\Delta(\widehat{\He_{t}}-\He_{0})=-\div((A_{t}-\Id)\nabla \widehat{\He_{t}})$, from classical elliptic estimate (see \cite[Corollary 8.7 p 183]{GT01Ell}), we obtain:
\begin{align}
\|\widehat{\He_{t}}-\He_{0}\|_{\sH^1(\Om)}&\leq C\|(A_{t}-\Id)\nabla\widehat{\He_{t}}\|_{\sL^2(\Om)}+C\|\widehat{\He_{t}}-\He_{0}\|_{\sH^{1/2}(\partial\Om)}\notag\\
&\leq  C\|A_{t}-\Id\|_{\sL^\infty(\Om)}\|\nabla \widehat{\He_{t}}\|_{\sL^2(\Om)}+C\left\|\left((\widehat{\partial_{\n_{t}}v_{t}})\widehat{\alpha_{t}}-\partial_{\n }v_{0}\right)h\right\|_{\sH^{1/2}(\partial \Om)}\notag\\
&\leq C\|h\|_{\sW^{1,\infty}(\partial\Om)}\left(\|\nabla \widehat{\He_{t}}-\nabla \He_{0}\|_{\sL^2(\Om)}+\|\nabla \He_{0}\|_{\sL^2(\Om)}\right)\\
&\hspace{2cm}+C\|\widehat{(\partial_{\n_{t}}v_{t}})\widehat{\alpha_{t}}-\partial_{\n }v_{0}\|_{\sW^{1-1/p,p}(\partial\Om)}\|h\|_{\sH^{1/2}(\partial \Om)}\notag\\
&\leq C\|h\|_{\sW^{2,p}(\partial \Om)}\left(\|\widehat{\He_{t}}-\He_{0}\|_{\sH^1(\Om)}+\|h\|_{\sH^{1/2}(\partial\Om)}\right).\label{eq:estimate:harm}
\end{align}

Here we used that $\|\nabla \He_{0}\|_{\sL^2(\Om)}=\|\He_{0}\|_{\sH^{1/2}(\partial\Om)}\leq C\|h\|_{\sH^{1/2}(\partial\Om)}$, Lemmas \ref{estimations:geometriques} and \ref{regularite:vp:fp}, and the following estimate of a product norm in $\sH^{1/2}$:
\begin{equation}\label{eq:produits}
\|uv\|_{\sH^{1/2}(\partial\Om)}\leq C\|u\|_{\sW^{s,p}(\partial\Om)}\|v\|_{\sH^{1/2}(\partial\Om)}
\end{equation}
if $(d-1)/p<s\leq 1$.
 One can find this inequality in \cite[Theorem 2 p 177]{RS} for functions on $\R^{d-1}$, with the condition $(d-1)/p<s$; using smooth maps between $\partial\Om$ and $\R^{d-1}$ we obtain \eqref{eq:produits} if in addition $s\leq 1$. We apply it here to $s=1-1/p$ which is valid as $p>d$.
Equation \eqref{eq:estimate:harm} leads to the first estimate 
$$\|\widehat{\He_{t}}-\He_{0}\|_{\sH^1(\Om)}\leq C\|h\|_{\sW^{2,p}(\partial\Om)}\|h\|_{\sH^{1/2}(\partial\Om)}$$
as soon as $\|h\|_{\sW^{2,p}(\partial\Om)}\leq 1/(2C)$.

\emph{3. Estimates on the variation of $w_{t}$.} 
We look at the PDE satisfied by $\widehat{w_{t}}-w_{0}$:
\begin{equation*}
\left\{
\begin{array}{rcl}
 (-\Delta -\lambda_{1}(0))(\widehat{w_{t}}-w_{0})  &=&\Big[(-\Delta-\lambda_{1}(0))-(-\mathcal{L}_{t}-\lambda_{1}(t))\Big](\widehat{w_{t}})-\lambda_{1}(t)\widehat{\pi_{t}\He_{t}}+\lambda_{1}(0)\pi_{0}\He_{0} \text{ in }\Omega, \\[10pt]
\widehat{w_{t}}-w_{0}& =&0\text{ on }\partial\Omega,\\[5pt]
\end{array}
\right.
\end{equation*}
and we know that $(-\Delta-\lambda_{1}(0))$ is an isomorphism on $\{v_{0}\}^\perp$. Therefore
$$\|\widehat{w_{t}}-w_{0}-\gamma_{t}v_{0}\|_{\sH^1(\Om)}\leq C\|(A_{t}-\Id)\nabla\widehat{w_{t}}\|_{\sL^2(\Om)}+|\lambda_{1}(t)-\lambda_{1}(0)|\|\widehat{w_{t}}\|_{\sL^2(\Om)}+\|\lambda_{1}(t)\widehat{\pi_{t}\He_{t}}-\lambda_{1}(0)\pi_{0}\He_{0}\|_{\sL^2(\Om)}$$
where $\gamma_{t}$ is chosen so that $\widehat{w_{t}}-w_{0}-\gamma_{t}v_{0}\in \{v_{0}\}^\perp$. From there and using the previous step, we obtain
$$\|\widehat{w_{t}}-w_{0}-\gamma_{t}v_{0}\|_{\sH^1(\Om)}\leq C\|h\|_{\sW^{2,p}(\partial\Om)}\left(\|\widehat{w_{t}}\|_{\sH^1(\Om)}+\|h\|_{\sH^{1/2}(\partial\Om)}\right).$$
But we have:
$$|\gamma_{t}|=\left|\int_{\Om}(\widehat{w_{t}}-w_{0})v_{0}\right|=\left|\int_{\Om}\widehat{w_{t}}\big[\widehat{v_{t}}J_{t}-v_{0}\big]\right|\leq C\|h\|_{\sW^{2,p}(\partial\Om)}\|\widehat{w_{t}}\|_{\sL^2(\Om)},$$
leading to
$$\|\widehat{w_{t}}-w_{0}\|_{\sH^1(\Om)}\leq C\|h\|_{\sW^{2,p}(\partial\Om)}\left(\|\widehat{w_{t}}\|_{\sH^1(\Om)}+\|h\|_{\sH^{1/2}(\partial\Om)}\right)\leq C\|h\|_{\sW^{2,p}(\partial\Om)}\left(\|\widehat{w_{t}}-w_{0}\|_{\sH^1(\Om)}+\|w_{0}\|_{\sH^1(\Om)}+\|h\|_{\sH^{1/2}(\partial\Om)}\right).$$
Using now 
$\|w_{0}\|_{\sH^1(\Om)}\leq C\|\He_{0}\|_{\sL^2(\Om)}\leq C\|h\|_{\sH^{1/2}(\partial\Om)}$ and again that $\|h\|_{\sW^{2,p}(\partial\Om)}$ is small enough, this leads to
$$\|\widehat{w_{t}}-w_{0}\|_{\sH^1(\Om)}\leq C\|h\|_{\sW^{2,p}(\partial\Om)}\|h\|_{\sH^{1/2}(\partial\Om)}$$
and concludes the proof of this lemma.
\end{proof}

\noindent{\bf Proof of Theorem \ref{th:CHX}:}
We deal separately with the terms of the decomposition \eqref{decompostion:derivee:seconde:directionnelle:a:t}: 

\noindent\emph{Estimate of $\mathcal{T}_{1}(t)-\mathcal{T}_{1}(0)$.} We first observe that
$$\mathcal{T}_{1}(t)=\int_{\Om_{t}}|\nabla v_{t}'|^2-\lambda_{1}(t)\int_{\Om_{t}}v_{t}^2,$$
and also that $\|v_{0}'\|_{\sH^1(\Om)}\leq \|w_{0}\|_{\sH^1(\Om)}+\|\He_{0}\|_{\sH^1(\Om)}\leq C\|h\|_{\sH^{1/2}(\partial\Om)}$.
Therefore using Lemma \ref{lemme:variation:derivee:fonction:propre}, we get
\begin{equation}
\left|\int_{\Omega_{t}} |\nabla v'_{t}|^{2} - \int_{\Omega_{0}} |\nabla v'_{0}|^{2} \right|=\left| \int_{\Omega} (A_{t}-\Id)|\nabla\widehat{v_{t}'}|^2+\nabla(\widehat{v_{t}'}-v'_{0})\cdot\nabla(\widehat{v_{t}'}+v_{0}')\right|\leq C \|h\|_{\sW^{2,p}(\partial\Omega)}\|h\|_{\sH^{1/2}(\partial\Omega)}^2
\end{equation}
and
\begin{eqnarray*}
\left|\lambda_{1}(t)\int_{\Omega_{t}} | v'_{t}|^{2} - \lambda_{1}(0) \int_{\Omega_{0}}|v'_{0}|^{2} \right|&=&\left|(\lambda_{1}(t)-\lambda_{1}(0)) \int_{\Omega_{t}}|v'_{t}|^{2}+\lambda_{1}(0)\int_{\Omega_{0}} (J_{t}-1) |\widehat{v_{t}'}|^{2} + (\widehat{v_{t}'} -v'_{0}) (\widehat{v_{t}'} +v'_{0})\right|\\
&\leq &C  \|h\|_{\sW^{2,p}(\partial\Omega)}  \|h\|_{\sH^{1/2}(\partial\Omega)}^2.
\end{eqnarray*}
\emph{Estimate of $\mathcal{T}_{2}(t)-\mathcal{T}_{2}(0)$.} After change of variable, we have
$\mathcal{T}_{2}(t)=\int_{\partial\Omega} \sigma_{t} h^2$
where
$$\sigma_{t}=(\widehat{\partial_{\n_{t}}v_{t}})^2 \left[ \widehat{H_{t}}\widehat{\alpha_{t}}^2-\widehat{\B_{t}}(\widehat{\beta_{t}},\widehat{\beta_{t}})-2\widehat{\nabla_{\tau_{t}}(\alpha_{t})}\cdot\widehat{\beta_{t}}\right]J_{\partial\Om}(t)$$
and from Lemmas \ref{estimations:geometriques} and \ref{regularite:vp:fp}, we easily get $\|\sigma_{t}-\sigma_{0}\|_{\sL^p(\partial\Om)} \leq C \|h\|_{\sW^{2,p}(\partial\Om)}.$
Notice that the control holds only in $\sL^p$ and not in $\sL^{\infty}$ as in  \cite{Dambrine} or \cite[Appendix]{BdPV}, hence we do not obtain a control with the $\sL^2$ norm of $h$. 
However, by H\"older inequality, it comes
$|\mathcal{T}_{2}(t)-\mathcal{T}_{2}(0)|\leq \|\sigma_{t}-\sigma_{0}\|_{\sL^p} \|h\|_{\sL^{\tilde p}}^2$ for any $\tilde p \geq 2p/(p-1)$. Since $\|h\|_{\sL^{\tilde p}}\leq C \|h\|_{\sH^{1/2}}$ when $\tilde p < 2d/(d-1)$ by Sobolev embeddings, such a $\tilde p$ can be chosen provided $p>d$. Then, it holds
 $$|\mathcal{T}_{2}(t)-\mathcal{T}_{2}(0)|\leq \|\sigma_{t}-\sigma_{0}\|_{\sL^p} \|h\|_{\sH^{1/2}}^2\leq   C\|h\|_{\sW^{2,p}} \|h\|_{\sH^{1/2}}^2 .$$
 \noindent \emph{Estimate of $\mathcal{T}_{3}(t)-\mathcal{T}_{3}(0)$.} After change of variable, we have
$\mathcal{T}_{3}(t)=\int_{\partial\Omega} \rho_{t}\cdot(\nabla_{\widehat{\tau_{t}}} h) h$
where
$$\rho_{t}=(\widehat{\partial_{\n_{t}}v_{t}})^2 \widehat{\alpha_{t}}\widehat{\beta_{t}}J_{\partial\Om}(t),\;\;\;\;\;\textrm{ and }\;\;\;\;\;\nabla_{\widehat{\tau_{t}}}h=\nabla h-(\nabla h\cdot \widehat{\n_t})\widehat{\n_{t}}$$
and we obtain (recall that $\nabla h\cdot\n=0$):
\begin{eqnarray}
\left|\mathcal{T}_{3}(t)-\mathcal{T}_{3}(0)\right|&\leq&\left|\int_{\partial\Om}\rho_{t}\cdot\left(\nabla_{\widehat{\tau_{t}}}h-\nabla_{\tau}h\right) h\right|+\left|\int_{\partial\Om}\left(\rho_{t}-\rho_{0})\cdot\nabla_{\tau}h\right) h\right|\\
&\leq&\|\nabla_{\widehat{\tau_{t}}}h-\nabla_{\tau}h\|_{\sH^{-1/2}}\|\rho_{t}h\|_{\sH^{1/2}}+\|(\rho_{t}-\rho_{0})h\|_{\sH^{1/2}}\|\nabla_{\tau}h\|_{\sH^{-1/2}}\\
&\leq&\|\nabla h\cdot(\widehat{\n_{t}}-\n)\|_{\sH^{-1/2}}\|\rho_{t}h\|_{\sH^{1/2}}+\|(\rho_{t}-\rho_{0})h\|_{\sH^{1/2}}\|h\|_{\sH^{1/2}}\label{eq:estT3}
\end{eqnarray}
In addition to \eqref{eq:produits}, we also have from \cite[Theorem 2 p 173]{RS} (see the comments on \eqref{eq:produits}):
\begin{equation}\label{eq:produits2}
\|uv\|_{\sH^{-1/2}(\partial\Om)}\leq C\|u\|_{\sW^{s,p}(\partial\Om)}\|v\|_{\sH^{-1/2}(\partial\Om)}
\end{equation}
if $\max\{1/2,(d-1)/p\}<s\leq 1$. Using again Lemmas \ref{estimations:geometriques1}, \ref{estimations:geometriques} and \ref{regularite:vp:fp}, we get
$$\|\rho_{t}-\rho_{0}\|_{\sW^{1-1/p,p}(\partial\Om)} \leq C \|h\|_{\sW^{2,p}(\partial\Om)}, \;\;\;\;\;\|\widehat{\n_{t}}-\n_{0}\|_{\sW^{1,p}(\partial\Om)} \leq C \|h\|_{\sW^{2,p}(\partial\Om)},$$
which combined with \eqref{eq:estT3}, \eqref{eq:produits} and \eqref{eq:produits2}, concludes the estimate of this term and hence the proof.
 \qed

\section{Applications}\label{sect:app}

\subsection{Retrieving some examples from the literature}{\label{ssect:literature}

In this paragraph, we apply our results  to retrieve previous results from the literature:\\

\noindent{\bf Isoperimetric inequalities:}
According to the previous sections, the perimeter satisfy conditions {\bf(C}$\!\!~_{\sH^{1}}${\bf)} and {\bf(IT}$\!\!~_{\sH^{1},\sW^{1,\infty}}${\bf)} at any smooth enough set, and in particular for the ball. Moreover, as shows Section \ref{ssect:balls},
we have
$$\ell_{1}[P](B_{1})=(d-1)\ell_{1}[\Vol](B_{1}), \;\;\;\;\;\textrm{ and }\;\;\;\;\;\;\ell_{2}[P-(d-1)\Vol](B_{1})(\varphi,\varphi)= \sum_{k=0}^{\infty}\sum_{l=1}^{d_{k}} (k-1)(k+d-1) \ \alpha_{k,l}(\varphi)^2.$$
Moreover, $\varphi\in T(\partial B_{1})$ if and only if $\alpha_{0,1}(\varphi)=\alpha_{1,i}(\varphi)=0$ for $i\in\{1,\ldots,d\}$. Therefore
 $B_{1}$ is a critical and strictly stable shape for $P$ under volume constraint, and up to translations: Theorem \ref{th:main:contrainte} applies, and we retrieve Fuglede's result from \cite{Fuglede} about nearly spherical domains.\\

Recently in \cite{Neumayer}, different improved versions (even with a better distance than the Fraenkel asymmetry for $d_{1}$ in \eqref{eq:stabJ}) of the quantitative isoperimetric inequality has been achieved for the anisotropic perimeter
$$P_{f}(\Om)=\int_{\partial\Om}f(\n_{\partial\Om})$$ where $f:\R^d\to\R_{+}$ is a convex positively 1-homogeneous function, whose minimizer under volume constraint is an homothetic version of the Wulff shape $K=\{f_{*}<1\}$ where $f_{*}$ is the gauge function of $f$. In particular in \cite[Theorem 1.3 and Section 4]{Neumayer} focused on the case where $K$ is assumed to be $\sC^2$ and uniformly convex, a strategy based on the second variation is used: the author proves in \cite[Lemma 4.1]{Neumayer} that $P_{f}$ satisfies conditions {\bf(C}$\!\!~_{\sH^{1}}${\bf)} and {\bf(IT}$\!\!~_{\sH^{1},\sW^{1,\infty}}${\bf)}. Therefore, this falls into the hypothesis of our Theorem \ref{th:main:contrainte}, so if we prove that $K$ satisfies \eqref{eq:stable}, then we retrieve \cite[Proposition 1.9]{Neumayer} (we assumed the shape to be $\sC^3$, but here this can be reduced to $\sC^2$, as noticed in Remark \ref{rk:reg}).
It is interesting to notice though that in order to show that $K$ satisfies \eqref{eq:stable}, the author in \cite{Neumayer} uses the quantitative Wulff isoperimetric inequality from \cite{FigMagPra} (obtained with optimal transport method). 
Therefore, up to our knowledge, there is no proof ``from scratch'' of the quantitative anisotropic isoperimetric inequality using a result similar to Theorem \ref{th:main:contrainte}.\\

\noindent{\bf The Ohta-Kawasaki model:}
In  \cite{AFM}, both steps of the strategy described page \pageref{page:strategy} are achieved in order to deal with the following functional, formulated in $\T^N=(\R/\Z)^N$ and which includes a non-local term:
\begin{equation}\label{eq:AFMfunction}
J(\Om)=P_{\T^N}(\Om)+\gamma G(\Om)
\;\;\;\;\textrm{ where } 
G(\Om)=\int_{\T^N}|\nabla w_{\Om}|^2\;\;\;\;\textrm{ and }\left\{\begin{array}{cll}-\Delta w_{\Om}&=& \mathbbm{1}_{\Om}-\mathbbm{1}_{\Om^c}-m \;\textrm{ in }\T^N\\[2mm] \displaystyle{\int_{\T^N}w_{\Om} }&=&0\end{array}\right.
\end{equation}
where $m=|\Om|-|\Om^c|\in(-1,1)$ is fixed. Again, there is an invariance with translation and a volume constraint.

In order to handle the first step of the strategy, the authors in \cite{AFM} prove a stability result for the $\sW^{2,p}$-topology, for $p$ large enough. 
The strategy is very similar to \cite{Dambrine}, but in the framework of $\sW^{2,p}$-spaces rather than $\sC^{2,\alpha}$-spaces. Note that this difference in the choice of spaces is not just a detail as it is relevant for the second step of the strategy when proving stability in an $\sL^1$-neighborhood as it is done in \cite[Section 4]{AFM}: their regularization procedure needs to allow discontinuity of the mean curvature, see equation (4.9) in the proof of \cite[Theorem 4.3]{AFM}. 
From the computations of \cite{Choksi-Sternberg}, we obtain
$$\ell_{1}[G](\Om)(\varphi)=4\int_{\partial\Om} w_{\Om}\varphi,$$
$$\ell_{2}[G](\Om)(\varphi,\varphi)=8\int_{\T^N} |\nabla z_{\varphi}|^2 dx+4\int_{\partial\Om} (\partial_{\n} w_{\Om}+H)\varphi^2,
\textrm{ where }-\Delta z_{\varphi}=\varphi\mathcal{H}^{N-1}\lfloor\partial\Om$$
therefore $G$ satisfies {\bf(C}$\!\!~_{\sH^{1/2}}${\bf)} and $J$ satisfies {\bf(C}$\!\!~_{\sH^{1}}${\bf)}, the dominant term being contained in the perimeter term. As we have seen that the perimeter satisfies {\bf(IT}$\!\!~_{\sH^{1},\sW^{1,\infty}}${\bf)} condition, it just remains to handle functional $G$, which is proven to satisfy {\bf(IC}$\!\!~_{\sH^{1},\sW^{2,p}}${\bf)} for $p>d$ in \cite{AFM}. Therefore Theorem \ref{th:main:contrainte} applies, and we retrieve \cite[Theorem 3.9]{AFM}.\\

\noindent{\bf The Faber-Krahn inequality:}
In \cite{BdPV} (see also \cite{FuscoZhang}) a quantitative version of the Faber-Krahn inequality is achieved, using again the two steps described page \pageref{page:strategy}: in order to achieve the first step, they use the Kohler-Jobin inequality (\cite{Kohler}), which implies that the Faber-Krahn deficit is controlled by the deficit of the Dirichlet energy $E$. We show here that it is possible to achieve this step without this ``trick'': we have seen that $\lambda_{1}$ satisfies
{\bf(C}$\!\!~_{\sH^{1/2}}${\bf)} and {\bf(IC}$\!\!~_{\sH^{1/2},\sW^{2,p}}${\bf)} for $p>d$, and for any $\varphi\in \C^\infty(\partial B_{1})$ such that $\int_{\partial B_{1}}\varphi=0$, we have
$$\ell_{1}[\lambda_{1}](B_{1})=-\beta_{d}^2\ell_{1}[\Vol](B_{1}), \;\;\;\;\;\textrm{ and }\;\;\;\;\;\;\ell_{2}[\lambda_{1}+\beta_{d}^2\Vol](B_{1})(\varphi,\varphi)= 2\beta_{d}^2\sum_{k=0}^{\infty}\sum_{l=1}^{d_{k}} Q_{k} \ \alpha_{k,l}(\varphi)^2.$$
where (using \cite[Section 9.1.27, p 361]{AbramowitzStegun})
$$Q_{k}=\bessel \cfrac{J'_{k+d/2-1}(\bessel)}{J_{k+d/2-1}(\bessel)}+\frac{d}{2}=k+d-1- \bessel \cfrac{J_{k+d/2}(\bessel)}{J_{k+d/2-1}(\bessel)}=\bessel \cfrac{J_{k+d/2-2}(\bessel)}{J_{k+d/2-1}(\bessel)}-k+1.$$
With the last formula, we easily notice that $Q_{1}=0$. The sign of $Q_{k}$ can be obtained using  \cite[section 6.5 page 133]{Polya-Szego} (done when $d=2$, but as noticed in \cite{Henry}, valid for any $d$): indeed, their computations imply
$$\bessel \cfrac{J'_{k+d/2-1}(\bessel)}{J_{k+d/2-1}(\bessel)}\geq k-d/2-1, \forall n\in\N^*,$$
which leads to $\forall k\geq 2, Q_{k}\geq k-1.$
Therefore Theorem \ref{th:main:contrainte} applies, and we retrieve a Faber-Krahn quantitative inequality in a $\sW^{2,p}$-neighborhood of the ball.
}

\subsection{Examples with competition}
\label{subsection:applications1}
In this section, $B$ is a ball, $X=\sW^{2,p}(\partial B)$ for $p>d$ and we denote  for $\eta>0$ (see \eqref{eq:distance} for a definition of $d_{X}$): 
\begin{equation}\label{eq:voisinage}
\mathcal{V}_{\eta}=\{\Om, d_{X}(\Om,B)\leq \eta\textrm{ and }|\Om|=|B|\}.
\end{equation}
Combining Theorem \ref{th:main:contrainte} to the computations from Section \ref{ssect:shapederivative}, we easily obtain the following result:

\begin{proposition}\label{prop:applications}
There exists $\gamma_{0}\in(0,\infty)$ such that for every $\gamma\in[-\gamma_{0},\infty)$, 
 there exists $\eta=\eta(\gamma)>0$ and $c=c(\gamma)>0$ such that for every $\Om\in \mathcal{V}_{\eta}$,
 $$(P+\gamma E)(\Om)\;\geq \;(P+\gamma E)(B)+cd_{\sH^{1}}(\Om,B)^2,\;\;\;(P+\gamma \lambda_{1})(\Om)\;\geq \;(P+\gamma \lambda_{1})(B)+cd_{\sH^{1}}(\Om,B)^2$$
$$
(E+\gamma \lambda_{1})(\Om)\;\geq \;(E+\gamma \lambda_{1})(B)+cd_{\sH^{1/2}}(\Om,B)^2,\;\;\;(\lambda_{1}+\gamma E)(\Om)\;\geq \;(\lambda_{1}+\gamma E)(B)+cd_{\sH^{1/2}}(\Om,B)^2.
$$
\end{proposition} 

\noindent{\bf Proof of Proposition \ref{prop:applications}}:
We show that we can apply Theorem \ref{th:main:contrainte} can be applied to $\Om^*=B$ and 
$$J\in\{P+\gamma E,P+\gamma \lambda_{1},E+\gamma\lambda_{1},\lambda_{1}+\gamma E)\}.$$
It is shown in Sections \ref{ssect:coercivity} and \ref{section:Chs} that $(P,E,\lambda_{1})$ satisfy {\bf(C}$\!\!~_{\sH^{s_{2}}}${\bf)} and {\bf(IT}$\!\!~_{\sH^{s_{2}},X}${\bf)} for suitable values of $s_{2}$, and with Lemmata \ref{diagonalisation:hessiennes} and \ref{continuite:derivee:seconde} we easily check that the ball is a critical and strictly stable domain for $J$ under volume constraint and up to translations, either if $\gamma\geq 0$ or if $\gamma<0$ is small enough.\qed
\begin{corollary}\label{cor:deficit}
With the same notations as in Proposition \ref{prop:applications}, we have, with $\eta_{0}=\eta(\gamma_{0})$:
\begin{eqnarray*}
\forall\Om\in \mathcal{V}_{\eta_{0}},
&\displaystyle{\frac{P(\Om)-P(B)}{E(\Om)-E(B)}\;\geq \gamma_{0},\;\;\;\;\;\;\frac{P(\Om)-P(B)}{\lambda_{1}(\Om)-\lambda_{1}(B)}\;\geq \gamma_{0}}\\[3mm]
&\displaystyle{\gamma_{0}\leq \frac{\lambda_{1}(\Om)-\lambda_{1}(B)}{E(\Om)-E(B)}\;\leq \gamma_{0}^{-1}}.
\end{eqnarray*}
\end{corollary}

\begin{remark}
In \cite{Nitsch}, the second inequality in Corollary \ref{cor:deficit} is also investigated, but we provide here a uniform neighborhood so that this estimate applies.
 We also refer to \cite{Polya-Szego} for some result of this kind.
\end{remark}
\begin{remark}
To the contrary to the last two-sided inequality, it is not possible to bound the first two ratio from above. Indeed, for every $\gamma\in(0,\infty)$, there exists $\Om_{\gamma}=(\Id+\theta_{\gamma})(B)$ of class $\sC^{\infty}$ such that
$$|\Om_{\gamma}|=|B|, \;\;\|\theta_{\gamma}\|_{\sW^{2,p}(\R^d)}\leq \gamma^{-1}\textrm{ and }{ \displaystyle{\frac{P(\Om)-P(B)}{E(\Om)-E(B)}\;> \gamma}}.$$
This is due to the fact that the functionals $P$ and $(E,\lambda_{1})$ satisfy conditions {\bf(C}$\!\!~_{\sH^{s_{2}}}${\bf)} for different values of $s_{2}$.
\end{remark}

\noindent{\bf Explicit constants:} We want to go further and compute explicit numbers $\gamma$ such that the inequalities of Proposition \ref{prop:applications} holds. To simplify the expressions, we restrict ourselves to the case of the unit ball. In the first two cases, we find the optimal constant, see Remark \ref{rk:optimal} about the other cases.

\begin{proposition}\label{prop:constante:optimale}
Using notations of Proposition \ref{prop:applications} and $\beta_{d}$ defined in \eqref{definition:gammad},
\begin{description}
\item[(i)] if $\gamma>-(d+1)d^2$, then $B_{1}$ is a local strict minimizer of $P+\gamma E$. 
Moreover, when $\gamma=-(d+1)d^2$, the second derivative of the Lagrangian cancels in some directions and when $\gamma<-(d+1)d^2$, the ball is a saddle shape for $P+\gamma E$.
\item[(ii)] if $\gamma>-  \ \cfrac{d(d+1)}{2\beta_{d}^2(\bessel^2-d)} $, \Bk  then $B_{1}$ is a local strict minimizer of $P+\gamma \lambda_{1}$. 
Moreover, when $\gamma=-\cfrac{d(d+1)}{2\beta_{d}^2(\bessel^2-d)}$, the second derivative of the Lagrangian cancels in some directions and when $\gamma<-\cfrac{d(d+1)}{2\beta_{d}^2(\bessel^2-d)}$, the ball is a saddle shape for $P+\gamma \lambda_{1}$.
\item[(iii)] if $\gamma>-\cfrac{1}{d^2  (d+1)\Bk \beta_{d}^2} $, then $B_{1}$ is a local strict minimizer of $E+\gamma \lambda_{1}$.
\item[(iv)] if $\gamma>- \beta_{d}^2 d^2 $, then $B_{1}$ is a local strict minimizer of $ \lambda_{1}+\gamma E$. 
\end{description}
\end{proposition}

\begin{remark}\label{rk:optimal}
In the cases (iii) and (iv), the constants we compute are not optimal, in particular we do not claim the ball is a saddle point once we go beyond the computed value. Nevertheless computing the optimal value only requires to compute $\sup_{k\geq 2}\tau'_{k}$ and $\sup_{k\geq 2}\tau''_{k}$ (see the notations in the proof below) as it is done in the cases (i) and (ii). As it is seen in the second case (ii) handled by Nitsch in \cite{Nitsch}, these computations can be rather technical. Let us notice also that we simplify the expression of the optimal constant given by Nitsch.
\end{remark}
\noindent {\bf Proof of Proposition  \ref{prop:constante:optimale}:}\\
\noindent {\bf (i)}
We first compute the Lagrange multiplier $\mu(t)$ associated to the volume constraint at $B_{1}$: it is defined as $\ell_{1}[P+tE)+\mu(t) \Vol]=0$ that is from the expression of the shape gradients of $\Vol$, $P$ and $E$:
$$\mu(t)=\cfrac{1}{2d^2} \ t \ -(d-1).$$
Let us now turn our attention to hessian of the function $P+tE+\mu(t)\Vol$ on the balls $B_{1}$. 
As a consequence of Lemma \ref{diagonalisation:hessiennes}, the shape hessian of the lagrangian $P+tE+\mu(t) \Vol$ at balls is
$$\ell_{2}[P+tE+\mu(t)\Vol](B_{1}).(\varphi,\varphi)=\sum_{k=0}^{\infty} c_{k}(t)\sum_{l=1}^{d_{k}}  \alpha_{k,l}(\varphi)^2$$ where we have set
$$c_{k}(t)= k^2+\left[ (d-2)+\cfrac{1}{d^2} \ t\right] \ k - \left[(d-1) + \cfrac{1}{d^2} \ t \right] =(k-1) \left[ k+(d-1)+\cfrac{1}{d^2} \ t \right].$$
Therefore, the hessian of the Lagrangian $\ell_{2}[P+tE+\mu(t) \Vol](B_{1})$ is coercive in $T(\partial B_{1})$ if and only if $t$ solves the inequalities 
$$ k+(d-1)+\cfrac{1}{d^2} \ t>0$$ for all $k\geq 2$. Of course, it suffices to solves that inequality in the special case $k=2$ that provides $t>-(d+1)d^2$.\\

\noindent{{\bf (ii)}}
With the same notions as in (i) with $P+t\lambda_{1}+\mu(t)\Vol$, we obtain : 
$$\mu(t)=\beta_{d}^2 \ t \ -(d-1), \;\;\;\;\;c_{k}(t)= k^2+(d-2+t\beta_{d}^2) k -(d-1)+t\beta_{d}^2 \left[d-1- \bessel \cfrac{J_{k+d/2}(\bessel)}{J_{k-1+d/2}(\bessel)}\right].$$
\Bk We introduce  the sequences $a_{k}=J_{k-1+d/2}(\bessel)$ and $b_{k}=a_{k+1}/a_{k}$   so that: 
$$c_{k}(t)= k^2+(d-2) k -(d-1)+2t\beta_{d}^2 \left[k+d-1- \bessel b_{k}\right].$$
 For a given integer $k\geq 2$, $c_{k}(t)>0$ holds when $t>\tau_{k}$ defined as
$$\tau_{k}=-\cfrac{(k-1)(k+d-1)}{2\beta_{d}^2(k+d-1-\bessel b_{k})}.$$
In order to obtain to find the optimal value of $t$ so that these inequalities are satisfied for every $k\geq 2$, we need to compute the supremum of $\{\tau_{k}, k\geq 2\}$.
It is proven by Nitsch in \cite[proof of Lemma 2.3, p 332]{Nitsch} that for all $k\geq 2, \tau_{k}\leq\tau_{2}$, so the ball is strictly stable if and only if $t>\tau_{2}$. We describe here how one can obtain a more explicit version of $\tau_{2}$: from the recurrence formula for Bessel function (\cite[section 9.1.27, p 361]{AbramowitzStegun})
$$ (2\nu/z)J_{\nu }(z) =J_{\nu -1}(z) +J_{\nu+1}(z)$$  
 applied to $\nu= k-1+d/2$ and $z=\bessel$, the sequences $a_{k}$ and $b_{k}$ satisfy the recurrence property
$$a_{k+1} = \cfrac{2(k-1)+d} {\bessel} \ a_{k}-a_{k-1} \text{ and } b_{k}= \cfrac{2(k-1)+d} {\bessel} \ -\ \cfrac{1}{b_{k-1}}$$
with the initial terms $a_{0}=0$ and $a_{1}=J_{d/2}(\bessel)$ so that $b_{1}=a_{2}/a_{1}=d/ \bessel $ (which explains  $c_{1}(t)= 0$ for any $t$, as known for the invariance by translations of all the involved functions).   Therefore, we have:
$$b_{2}=\cfrac{2+d}{\bessel}-\cfrac{\bessel}{d}=\frac{d(d+2)-\bessel^2}{d\bessel}
$$
and as a consequence, we obtain that 
$$\tau_{2}= -\ \cfrac{d(d+1)}{2\beta_{d}^2(\bessel^2-d)}.
$$

\Bk

\noindent{\bf (iii)} With the same notions as in (i) with $E+t\lambda_{1}+\mu(t)\Vol$, we obtain : 
$$\mu(t)=(1/d^2)+t \beta_{d}^2, \;\;\;\;\;c_{k}(t)= \left(\frac{1}{{d}^2}+t\beta_{d}^2\right) k -	\cfrac{1}{d^2}+t\beta_{d}^2 \left[d-1- \bessel b_{k}\right].$$
Again $c_{1}(t)=0$ and $c_{k}(t)> 0$ if and only if 
$$t> \tau'_{k} = -\cfrac{k-1}{d^2 \beta_{d}^2 (k+d-1-\bessel b_{k})}.$$
Using that $b_{1}\geq b_{k}>0$, we obtain
$$\tau'_{k} <-\cfrac{1}{d^2\beta_{d}^2} \ \cfrac{k-1}{k+d-1}   =  -\cfrac{1}{d^2\beta_{d}^2} \ \left(1- \cfrac{d}{k+d-1}\right)\leq  -\cfrac{1}{d^2(d+1)\beta_{d}^2}.$$
Therefore, if $t>-\cfrac{1}{d^2(d+1)\beta_{d}^2}$ then for any $k\geq 2$, $t> \tau'_{k}$, which leads to the result.\\

\noindent{\bf (iv)} With the same notions as in (i) with $\lambda_{1}+tE+\mu(t)\Vol$, we obtain : 
$$\mu(t)=(t/d^2)+ \beta_{d}^2, \;\;\;\;c_{k}(t)= \left(\frac{t}{{d}^2}+\beta_{d}^2\right) k -	\cfrac{t}{d^2}+\beta_{d}^2 \left[d-1- \bessel b_{k}\right].$$
We check $c_{1}(t)=0$, and $c_{k}(t)> 0$ if and only if 
$$t> \tau''_{k}=- \beta_{d}^2 d^2 \ \left(1 + \cfrac{d-\bessel b_{k}}{k-1}\right) .$$
  Using that $b_{1}\geq b_{k}>0$, we obtain $\tau''_{k}\leq - \beta_{d}^2 d^2,$ and therefore, if $t>- \beta_{d}^2 d^2$ then for any $k\geq 2$, $t>\tau''_{k}$, which leads to the result.
\qed

\section{Counterexample for non smooth perturbations}\label{ssect:non-stab}

We show in this section that even if the ball is a local minimum in a smooth neighborhood, it may not be a local minimum in a non-smooth neighborhood.

Consider 
 $\Om^*=B$ a ball of volume $V_{0}$. We have seen in Proposition \ref{prop:applications} that there is  a real number $\gamma_{0}\in(0,\infty)$ such that for every $\gamma\in(-\gamma_{0},\infty)$, $B$ is a stable local minimum for $P+\gamma E$. 

For $\gamma\geq 0$ this is not surprising. 
However, \emph{for $\gamma<0$, the fact that the ball is a local minimizer is no longer trivial}: there is a competition between the minimization of the perimeter and maximization the Dirichlet energy. If $\gamma$ small enough, our result shows  that $B$ is still a local minimizer in a $\sW^{2,p}$-neighborhood. Nevertheless, 
in that case $B$ is no longer a local minimizer in a $\sL^1$-neighborhood :

\begin{proposition}
Let $B$ be a ball. For every $\gamma<0$ and any $\eps>0$ one can find $\Om_{\eps}$ such that 
$$|\Om_{\eps}\Delta B|<\eps, \;\;\;|\Om_{\eps}|=|B|, \;\;\;\textrm{ and }\;\;\;(P+\gamma E)(\Om_{\eps})<(P+\gamma E)(B).$$
\end{proposition}

To prove this result, we use the idea of topological derivative:  it is well known that if one consider a small hole of size $\eps$ in the interior of a fixed shape, the energy will change at order $\eps^{d-2}$ if $d\geq 3$ and ${1}/{\log(\eps)}$ if $d=2$, which is strictly bigger than the change of perimeter which is of order $\eps^{d-1}$, and therefore will strictly decrease the energy $P+\gamma E$ when $\gamma<0$. For the sake of completeness, we provide a proof of this fact for a centered hole.\\[2mm]
\begin{proof}
We can assume without loss of generality (using translation and scaling properties) that $B=B_{1}$ is the centered ball of radius 1,   and we define $\Om_{\eps}=B_{1}\setminus B(0,\eps)$. Using that $\Delta u=\partial_{rr}u+\frac{d-1}{r}\partial_{r}u$ when $u$ is radial, the state function is:
$$u_{\Om_{\eps}}(r)=\frac{(\eps^{d-2}-\eps^d)r^{2-d}+\eps^d-1}{2d(\eps^{d-2}-1)}-\frac{r^2}{2d},\;\textrm{ if }d\geq 3$$
$$u_{\Om_{\eps}}(r)=\frac{1-\eps^2}{-4\log(\eps)}\log(r)+\frac{1-r^2}{4},\;\textrm{ if }d=2$$
and therefore 
\begin{eqnarray*}
\textrm{ if }d\geq 3, \;\;E(\Om_{\eps})&=&-\frac{1}{2}\int_{\Om_{\eps}}u_{\Om_{\eps}}=\left[\frac{d(1-\eps^2)^2\eps^{d-2}-2(1-\eps^d)^2}{8d^2(1-\eps^{d-2})}+\frac{1-\eps^{d+2}}{4d(d+2)}\right]P(B_{1})
\\
&=&\left[-\frac{1}{2d^2(d+2)}+\frac{d-2}{8d^2}\eps^{d-2}+o(\eps^{d-2})\right]P(B_{1}),
\end{eqnarray*}
\begin{eqnarray*}
\textrm{ if }d=2, \;\;E(\Om_{\eps})&=&-\frac{1}{2}\int_{\Om_{\eps}}u_{\Om_{\eps}}=\left[\frac{(1-\eps^2)}{-8\log(\eps)}(1-\eps^2(1-2\log(\eps)))-\frac{1}{16}(1-\eps^2+\frac{\eps^4}{2})\right]P(B_{1})
\\
&=&\left[-\frac{1}{16}-\frac{1}{8\log(\eps)}+o\left(\frac{1}{\log(\eps)}\right)\right]P(B_{1}).
\end{eqnarray*}
We now define $\widetilde{\Om_{\eps}}=\mu_{\eps}\Om_{\eps}$ where $\mu_{\eps}=(1-\eps^d)^{-1/d}$ so that
$$|\widetilde{\Om_{\eps}}|=|B_{1}|, \;\;\;\;P(\widetilde{\Om_{\eps}})-P(B_{1})=\left[\mu_{\eps}^{d-1}(1+\eps^{d-1})-1\right]P(B_{1})\sim_{\eps\to 0}\eps^{d-1}P(B_{1})$$
$$E(\widetilde{\Om_{\eps}})-E(B_{1})\sim_{\eps\to 0} \frac{(d-2)P(B_{1})}{8d^2}\eps^{d-2}>0, \;\;\textrm{ if }d\geq 3, \;\;\;E(\widetilde{\Om_{\eps}})-E(B_{1})\sim_{\eps\to 0} \frac{P(B_{1})}{-8\log(\eps)}>0, \;\;\textrm{ if }d=2$$
so that in both cases, for any negative $\gamma$,
$(P+\gamma E)(\Om_{\eps})-(P+\gamma E)(B_{1})<0 \textrm{ for small }\eps. $
\end{proof}

\paragraph{Acknowledgements.} This work was partially supported by the project ANR-12-BS01-0007 OPTIFORM financed by the French Agence Nationale de la Recherche (ANR).

\bibliographystyle{essai2}
\bibliography{references-5}
\bigskip

\noindent------------------------------------------------------------------

\noindent Marc Dambrine

\smallskip

\noindent CNRS / UNIV Pau \& Pays Adour / E2S UPPA, Laboratoire de Math\'ematiques et de leurs Applications de Pau - F\'ed\'eration IPRA, UMR 5142 64000, Pau, France

\smallskip

\noindent E-mail: \texttt{marc.dambrine@univ-pau.fr}
\newline\bigskip
\texttt{http://web.univ-pau.fr/\symbol{126}mdambrin/Marc{\_}Dambrine/Home.html}

\noindent Jimmy Lamboley

\smallskip

\noindent Sorbonne universit\'e, Institut Math\'ematiques de Jussieu-Paris Rive Gauche, CNRS, Univ Paris Diderot. \\Campus Pierre et Marie Curie, 4 place Jussieu, 75252 Paris Cedex 5, France

\smallskip

\noindent E-mail: \texttt{jimmy.lamboley@imj-prg.fr}

\noindent\texttt{https://webusers.imj-prg.fr/\symbol{126}jimmy.lamboley/}

\end{document}